\documentclass[onefignum,onetabnum]{siamart220329}

\usepackage{amsmath,amssymb,amsfonts}
\usepackage{graphicx}
\usepackage{epstopdf}
\usepackage{algorithmic}
\ifpdf
  \DeclareGraphicsExtensions{.eps,.pdf,.png,.jpg}
\else
  \DeclareGraphicsExtensions{.eps}
\fi

\usepackage{physics}
\usepackage{mleftright}
\usepackage{xcolor}
\usepackage{float}
\usepackage{tikz,subcaption}
\usetikzlibrary{decorations.pathmorphing,patterns}
\numberwithin{equation}{section}
\let\left\mleft
\let\right\mright
\def\d{ \mathrm{d}}

\allowdisplaybreaks

\newsiamremark{remark}{Remark}
\newsiamthm{assumption}{Assumption}
\newsiamremark{example}{Example}
\headers{Inverse Problem of Generalized Bathtub Model}{K. Huang, B. Jin, and Z. Zhou}

\title{On an Inverse Problem of the Generalized Bathtub Model of Network Trip Flows\thanks{
\funding{The work of B. Jin is supported by Hong Kong RGC General Research Fund (Project 14306423), and a start-up fund from The Chinese University of Hong Kong. The work of  Z. Zhou is supported by Hong Kong Research Grants Council (15303122) and an internal grant of Hong Kong Polytechnic University (Project ID: P0038888, Work Programme: 1-ZVX3).}}}

\author{Kuang Huang\thanks{Department of Mathematics, The Chinese University of Hong Kong, Shatin, New Territories, Hong Kong, P.R. China (\email{kuanghuang@cuhk.edu.hk}, \email{b.jin@cuhk.edu.hk}).}
\and Bangti Jin\footnotemark[2]
\and Zhi Zhou\thanks{Department of Applied Mathematics,
The Hong Kong Polytechnic University, Kowloon, Hong Kong, P.R. China (\email{zhizhou@polyu.edu.hk})}}

\usepackage{amsopn}

\begin{document}

\maketitle

% REQUIRED
\begin{abstract}
In this work, we investigate the generalized bathtub model, a nonlocal transport equation for describing network trip flows served by privately operated vehicles inside a road network. First, we establish the well-posedness of the mathematical model for both classical and weak solutions. Then we consider an inverse source problem of the model with model parameters embodying particular traffic situations. We establish a conditional Lipschitz stability of the inverse problem under suitable \textit{a priori} regularity assumption on the problem data, using a Volterra integral formulation of the problem. Inspired by the analysis, we develop an easy-to-implement numerical method for reconstructing the flow rates, and provide the error analysis of the method. Further we present several numerical experiments to complement the theoretical analysis.
\end{abstract}

% REQUIRED

% REQUIRED

\section{Introduction}

Transportation systems serve trips of passengers and goods and play a vital role in many economic and social activities. Mathematically describing transportation systems is complex due to the existence of interactions among passengers, vehicles, facilities, operators, and other stakeholders. Various mathematical models have been proposed to assist the plan, design, management, and operation of these complex systems (see, e.g., the recent monograph \cite{Jin:2021} for an overview), which has been a core task for transportation researchers and engineers alike.

For privately operated vehicles, the number of active trips can be converted to the number of running vehicles with given occupancies. Existing traffic flow models generally describe
the evolution of vehicular flows with privately operated vehicles and are often viewed as trip flow models, which can be roughly divided into three types depending on how a road network is treated, i.e., a discrete
network of roads and intersections (see, e.g., \cite{Hidas:2005,RoessMcShane:2010,Jin:2012}), 
continuum two-dimensional region (see, e.g., \cite{Beckmann:1952,HoWong:2006}), single unit with roads of the same traffic conditions. In the third type, a road network with vehicular
traffic flows is treated analogously to a bathtub with water flow, in which the detailed topology of a network or  spatial distribution of origins, destinations, routes and
roads are ignored, and only the lane-miles as well as the average speed-density relation (i.e., network fundamental diagram)
are relevant, and hence the name bathtub model. The bathtub model was first developed and analyzed by Vickrey \cite{Vickrey:1991,Vickrey:1994,Vickrey:2019,Vickrey:2020}. Distinctly, it can capture the interactions of dynamic travel demand and network supply effectively, and thus has been employed in several applications.
Small and Chu \cite{SmallChu:2003} used Vickrey's bathtub model to find the departure time user equilibrium in a network. Fosgerau \cite{Fosgerau:2015} employed the model to describe
deterministic distributions of trip distances, in which trips are ``regularly sorted'' such that shorter trips enter the network later
but exit earlier than longer ones (last-in-first-out). 
Arnott et al \cite{Arnott:2013,ArnottNaji:2016,ArnottBuli:2018} developed a bathtub model in which all trips are
assumed to have the same distance, and the dynamics of the number of active trips and the experienced travel times are modeled by delay-differential equations. Thus, the bathtub model represents an important class of mathematical models in traffic flow modeling.

In this work, we investigate the following \emph{generalized bathtub model} proposed in \cite{jin2020generalized}:
\begin{align}\label{eq:bathtub_model}
    \partial_t k(t,x) - V \left( \int_0^\infty k(t,y) w(y) \d y \right) \partial_x k(t,x) = f(t) \phi(t,x),
\end{align}
where $t$ represents time and $x$ represents the remaining trip distance, $k(t,x)$ is the density of active trips at time $t$ with a remaining distance $x$, $f(t)$ is the inflow rate, i.e., rate of trips entering the network at time $t$, and $\phi(t,x)$ is the inflow distribution with respect to $x$ at time $t$. The function $w=w(y)$ is a nonnegative kernel and the integral $\int_0^\infty k(t,y) w(y) \d y$ denotes the weighted average density of active trips on the network;
and $V=V(k)$ describes the relation between trips' average density and average speed.

The model \eqref{eq:bathtub_model} was  derived from the fundamental assumption that the velocity function $V=V(k)$ can capture the relation between all trips' densities and speeds on an aggregate level \cite{jin2020generalized}. 
Such a density-speed relation, and its corresponding density-flow relation $Q=Q(k)=k V(k)$, is also referred to as the {macroscopic fundamental diagram}.
This density-speed relation has been  justified by either empirical evidences \cite{johari2021macroscopic,geroliminis2008existence,batista2021identification} or analytical derivation \cite{daganzo2008analytical}. Moreover, its dependence on network topology has also been examined \cite{knoop2015network}.
There is a growing number of research works since \cite{daganzo2008analytical} that build models based on the macroscopic fundamental diagram. These models consider the number of active trips, their average trip distance, or the variation of their trip distances as primary variables and build ODEs to describe the relation of those variables. In contrast, the generalized bathtub model considers the distribution of active trips with respect to their remaining trip distances and it is formulated as a PDE with a nonlocal integral term. It is shown in \cite{jin2020generalized} that under the assumption that the active trips have an %negative 
exponential distribution over remaining trip distances, the generalized bathtub model reduces to the model proposed in \cite{daganzo2008analytical}.

In all these existing works, the kernel $w(y)$ is chosen to be
\[ w(y)=\frac1L\mathbf{1}_{[0,L]}(y), \quad y\in[0,\infty), \]
where $L$ is the total lane-miles of the network. It can capture the impacts of demand on traffic congestion, represented by the inflow and entering trips' distances, as well as the supply, represented by the network fundamental diagram. For trips served by privately operated vehicles, the generalized bathtub model is similar to the Lighthill-Whitham-Richards (LWR) model \cite{lighthill1955kinematic,richards1956shock} (see also \cite{payne1971model,aw2000resurrection,zhang2002non} for related continuum traffic flow models), as it is also derived from the fundamental diagram and the conservation law of vehicles. However, as a transport equation with a nonlocal speed, the generalized bathtub model admits no shock or rarefaction waves, since the characteristic curves do not interact with each other. 
The role of the generalized bathtub model for trip flow dynamics is the same as that of the LWR model for traffic flow dynamics.

The model \eqref{eq:bathtub_model} is equipped with the following initial condition:
\begin{align}\label{eq:initial_condition}
    k(0,x) = \bar{k}(x), \quad x\in[0,\infty).
\end{align}
Note that when the velocity $V$ stays nonnegative, there is only outflow at the boundary $x=0$ and no boundary condition is needed.
All inflows are described by the source term $f(t)\phi(t,x)$.

We contribute to the mathematical analysis of both direct problem and inverse problem associated with the generalized bathtub model \eqref{eq:bathtub_model}, with our main contributions summarized as follows. First, we establish the well-posedness of the direct problem under mild assumptions on the kernel $w$, employing a fixed point argument combined with the method of characteristics. Second, we investigate an inverse source problem to recover the inflow rates $f(t)$ from lateral boundary measurements $k(t,0)$ over a given time horizon $t\in[0,T]$. We show a conditional Lipschitz stability by reformulating the original problem as a nonlinear Volterra integral equation of the second kind and employing the method of successive approximations. 
Third, inspired by the analysis, we develop an easy-to-implement finite-difference method for numerically reconstructing the inflow rates $f$ and provide an error analysis of the method.
Fourth and last, we demonstrate the effectiveness of the reconstruction method and validate the theoretical findings through a series of numerical experiments. Additionally, we extend our numerical investigation to an inverse problem of recovering the time-independent inflow distribution $\phi$ from the lateral boundary data $k(\cdot,0)$. 
 
Now we put the work into the context of mathematical modeling for traffic flows. Nonlocal traffic flow models have been extensively studied over the last decade, and various well-posedness results have been obtained (see, e.g., \cite{BlandinGoatin:2016,bressan2019traffic,keimer2017existence}). Furthermore, due to the practical significance of traffic flow modeling, it is important to calibrate traffic flow models with field data.
One well established procedure is to estimate the macroscopic fundamental diagram, especially for first-order traffic flow models \cite{ZhongHe:2016,CoullonPokern:2022}, or nonlinear optimization with the dynamics described by first or high-order models \cite{Spiliopoulou:2014,Mohammadia:2021}; see \cite{WangSun:2022} for a state-of-the-art review. More recently machine learning techniques have also been integrated with traffic flow models for parameter estimation \cite{ShiMoDi:2021,ShiMoHuang:2022,HuangAgarwal:2022,ZhaoYu:2023}. However, none of these existing works has studied  analytically or numerically the generalized bathtub model, and the analytical study of inverse problems for traffic flow models remains very limited to this date. To the best of our knowledge, the present work contributes one first rigorous mathematical study to the topic. 

The rest of the paper is organized as follows. In \cref{sec:direct}, we discuss the well-posedness of the direct problem. Then in \cref{sec:inverse1}, we analyze the inverse source problem of recovering the inflow rates from boundary data. In \cref{sec:recon}, we propose an algorithm for recovering the inflow rates and discuss the error estimate. Finally in \cref{sec:exp}, we present several numerical experiments illustrating the feasibility of the reconstruction.

\section{Well-posedness of the direct problem}\label{sec:direct}

In this section, we establish the well-posedness of the direct problem \eqref{eq:bathtub_model}-\eqref{eq:initial_condition}. Throughout this section, we make the following assumptions on the nonlocal kernel $w$, the velocity function $V$, and the inflow rates $f$. Additionally, the notation $\norm{\,\cdot\,}_{\mathbf{L}^p}$ for $p\in[1,\infty]$ is used as shorthand for the $\mathbf{L}^p$ norm on $[0,\infty)$, and $\mathrm{TV}(\,\cdot\,)$ denotes the total variation on $[0,\infty)$.

\begin{assumption}\label{assm:1}
{\rm(i)} The nonlocal kernel $w\in\mathbf{L}^1\cap\mathbf{L}^\infty([0,\infty);[0,\infty))$ and it satisfies $\int_0^\infty w(y)\d y=1$; {\rm(ii)} the velocity function $V$ is Lipschitz continuous on $[0,\infty)$ and satisfies $V(k)\geq0$ for all $k\geq0$; and {\rm(iii)} the inflow rates $f\in\mathbf{L}^\infty([0,\infty);[0,\infty))$.
\end{assumption}

We first give the existence, uniqueness, and stability of strong solutions, assuming that both the initial data $\bar{k}$ and the inflow distribution $\phi$ are Lipschitz continuous. The proof is based on combining a fixed point argument with the method of characteristics, which has been widely used in the study of hyperbolic conservation and balance laws with nonlocal integrals inside the flux function; See, e.g., \cite{coron2010analysis,bressan2019traffic,colombo2019singular,keimer2017existence}. 
\begin{theorem}\label{thm:wellposedness_strong}
Suppose that the initial data $\bar{k}\in\mathbf{L}^\infty([0,\infty);[0,\infty))$ is Lipschitz continuous, that $\phi\in\mathbf{L}^\infty([0,\infty)\times[0,\infty);[0,\infty))$ is Lipschitz continuous in space and its Lipschitz constant is locally bounded on $t\in[0,\infty)$, and that \cref{assm:1} holds.
Then there exists a unique Lipschitz continuous function $k$ that satisfies problem \eqref{eq:bathtub_model}-\eqref{eq:initial_condition} pointwise for $(t,x)\in[0,\infty)\times[0,\infty)$.
Moreover, we have the following stability estimate:
\begin{align}\label{eq:ch2_stab_Linf}
    &\|{k^{(1)}(t,\cdot)-k^{(2)}(t,\cdot)}\|_{\mathbf{L}^\infty} \\
    &\quad \leq C(t) \left( \|{\bar{k}^{(1)} - \bar{k}^{(2)}}\|_{\mathbf{L}^\infty} + \sup\nolimits_{\tau\in[0,t]}\|{\phi^{(1)}(\tau,\cdot)-\phi^{(2)}(\tau,\cdot)}\|_{\mathbf{L}^\infty} \right), \notag
\end{align}
for $t\in[0,\infty)$, where $C(t)$ depends on $t,\norm{f}_{\mathbf{L}^\infty}, \norm{V'}_{\mathbf{L}^\infty}, \norm{\frac{\d}{\d x}\bar{k}^{(1)}}_{\mathbf{L}^\infty}$, and \\$ \sup\nolimits_{\tau\in[0,t]} \norm{\partial_x\phi^{(1)}(\tau,\cdot)}_{\mathbf{L}^\infty}$.
\end{theorem}
\begin{proof}
We break down the proof into three steps.\\
\textbf{(Step 1)} Fix a $T>0$ and consider the domain
\begin{align*}
    \mathcal{D} \doteq \mathbf{L}^\infty([0,T]\times[0,\infty);[0,\infty)).
\end{align*}
For any $k \in \mathcal{D}$, let $h = \Gamma [k]$ be the solution of the following linear Cauchy problem:
\begin{align*}
    &\partial_th(t,x) - v(t)\partial_xh(t,x) = f(t)\phi(t,x), \qquad h(0,x) = \bar{k}(x), \\
    &\mathrm{with} \quad v(t) = V \left( \int_0^\infty k(t,y) w(y) \d y \right).
\end{align*}
The characteristic curves of the problem are given by
\[ t(\tau) = \tau, \quad x(\tau) = x_0 - \xi(\tau), \quad \tau\in[0,\infty), \]
for any $x_0\geq0$, with
\begin{align}\label{eq:char_xi}
    \quad \xi(\tau) = \int_0^\tau v(\tilde{\tau}) \d\tilde{\tau}, \quad \tau\in[0,\infty).
\end{align}
Using the method of characteristics, we obtain
\begin{align*}
    h(t,x) = \bar{k}(x+\xi(t)) + \int_0^t f(\tau) \phi(\tau,x+\xi(t)-\xi(\tau)) \d\tau, \quad t\in[0,T],\, x\geq0.
\end{align*}
Given the Lipschitz continuity of $\bar k$ and $\phi$, it is straightforward to verify that $h\in\mathcal{D}$, and therefore, $\Gamma$ maps $\mathcal{D}$ into itself. Next we show that $\Gamma$ is a contraction mapping on $\mathcal{D}$.
For any $k_1,k_2\in\mathcal{D}$, we denote $h_1=\Gamma [k_1]$ and $h_2=\Gamma [k_2]$ and use $\xi_1$ and $\xi_2$ to represent the respective characteristic curves. Then we have
\begin{align*}
    &|h_1(t,x) - h_2(t,x)| \leq |\bar{k}(x+\xi_1(t)) - \bar{k}(x+\xi_2(t))| \\
    & \qquad \qquad + \int_0^t f(\tau) |\phi(\tau,x+\xi_1(t)-\xi_1(\tau)) - \phi(\tau,x+\xi_2(t)-\xi_2(\tau))| \d\tau \\
    & \qquad \leq \norm{\bar{k}'}_{\mathbf{L}^\infty} |\xi_1(t)-\xi_2(t)| + \norm{f}_{\mathbf{L}^\infty} M_{\phi} \int_0^t |\xi_1(t)-\xi_2(t)| + |\xi_1(\tau)-\xi_2(\tau)| \d\tau,
\end{align*}
for all $t\in[0,T]$ and $x\geq0$, where $ M_\phi \doteq \sup_{t\in[0,T]} \norm{\partial_x\phi(t,\cdot)}_{\mathbf{L}^\infty} $.
Moreover, by the definition of $\xi$ in \eqref{eq:char_xi} we have
\begin{align*}
    \sup\nolimits_{t\in[0,T]} |\xi_1(t) - \xi_2(t)|  \leq T\norm{V'}_{\mathbf{L}^\infty} \sup\nolimits_{t\in[0,T]} \norm{k_1(t,\cdot)-k_2(t,\cdot)}_{\mathbf{L}^\infty}.
\end{align*}
Consequently, we obtain
\begin{align*}
    &\sup\nolimits_{t\in[0,T]} \norm{h_1(t,\cdot) - h_2(t,\cdot)}_{\mathbf{L}^\infty} \\
    \leq& T\norm{V'}_{\mathbf{L}^\infty} \left( \norm{\bar{k}'}_{\mathbf{L}^\infty} + 2T \norm{f}_{\mathbf{L}^\infty} M_{\phi} \right) \sup\nolimits_{t\in[0,T]}\norm{k_1(t,\cdot)-k_2(t,\cdot)}_{\mathbf{L}^\infty}.
\end{align*}
This estimate implies that by choosing $T$ sufficiently small, $\Gamma$ is a contraction mapping on $\mathcal{D}$. Then by Banach contraction mapping theorem, there exists a unique $k\in\mathcal{D}$ such that $k=\Gamma [k]$. \\
\textbf{(Step 2)} Note that $k$ satisfies
\begin{align}\label{eq:char_sol}
    k(t,x) = \bar{k}(x+\xi(t)) + \int_0^t f(\tau) \phi(\tau,x+\xi(t)-\xi(\tau)) \d\tau, \quad t\in[0,T], \, x\geq0,
\end{align}
which together with the Lipschitz continuity of $\bar{k}$, $\phi$ and $\xi$ yields the Lipschitz continuity of $k$ in both space and time.
Thus it follows from the representation \eqref{eq:char_sol} that $k$ satisfies \eqref{eq:bathtub_model}-\eqref{eq:initial_condition} pointwise. Moreover, the representation \eqref{eq:char_sol} gives the following estimate on $\|{\partial_xk(t,\cdot)}\|_{\mathbf{L}^\infty}$:
\begin{align*}
        \|{\partial_xk(t,\cdot)}\|_{\mathbf{L}^\infty} \leq \|{\bar{k}'}\|_{\mathbf{L}^\infty} + t \|{f}\|_{\mathbf{L}^\infty} \sup\nolimits_{\tau\in[0,t]} \|{\partial_x\phi}(\tau,\cdot)\|_{\mathbf{L}^\infty},
    \end{align*}
which allows us to extend the solution to all times $t\in[0,\infty)$.\\
\textbf{(Step 3)} Now consider two sets of initial data and inflow distributions: $\bar{k}^{(i)}$ and $\phi^{(i)}$ for $i=1,2$, and denote the respective solutions by $k^{(i)}$, $i=1,2$.
Using the relations \eqref{eq:char_xi}-\eqref{eq:char_sol}, we derive that for $i=1,2$:
\begin{align}
    k^{(i)}(t,x) &= \bar{k}^{(i)}\left(x+\xi^{(i)}(t)\right) + \int_0^t f(\tau) \phi^{(i)}\left(\tau,x+\xi^{(i)}(t)-\xi^{(i)}(\tau)\right) \d\tau, \label{eq:ch2_stab_calc_1} \\
    \xi^{(i)}(\tau) &= \int_0^\tau V \left( \int_0^\infty k^{(i)}(\tilde{\tau},y) w(y) \d y \right) \d\tilde{\tau}. \label{eq:ch2_stab_calc_2}
\end{align}
Using the identity \eqref{eq:ch2_stab_calc_1}, we obtain 
\begin{align*}
    &\|{k^{(1)}(t,\cdot)-k^{(2)}(t,\cdot)}\|_{\mathbf{L}^\infty} \leq \|{\bar{k}^{(1)} - \bar{k}^{(2)}}\|_{\mathbf{L}^\infty} + t\|{f}\|_{\mathbf{L}^\infty} \sup_{\tau\in[0,t]}\|{\phi^{(1)}(\tau,\cdot)-\phi^{(2)}(\tau,\cdot)}\|_{\mathbf{L}^\infty} \\
    & \qquad\qquad + \norm{\frac{\d}{\d x}\bar{k}^{(1)}}_{\mathbf{L}^\infty} E(t,0) + \sup\nolimits_{\tau\in[0,t]} \norm{\partial_x\phi^{(1)}(\tau,\cdot)}_{\mathbf{L}^\infty}\|{f}\|_{\mathbf{L}^\infty} \int_0^t E(t,\tau) \d\tau,
\end{align*}
with
\begin{equation}\label{eq:E_t_tau}
    E(t,\tau) = \left|\xi^{(1)}(t)-\xi^{(1)}(\tau)-\xi^{(2)}(t)+\xi^{(2)}(\tau)\right|.
\end{equation}
Then by using \eqref{eq:ch2_stab_calc_2}, we obtain 
\begin{align*}
    E(t,\tau) \leq \norm{V'}_{\mathbf{L}^\infty}\int_{\tau}^t \|{k^{(1)}(\tilde{\tau},\cdot)-k^{(2)}(\tilde{\tau},\cdot)}\|_{\mathbf{L}^\infty} \d\tilde{\tau},
\end{align*}
and thus
\begin{align*}
    \|{k^{(1)}(t,\cdot)-k^{(2)}(t,\cdot)}\|_{\mathbf{L}^\infty} \leq& \|{\bar{k}^{(1)} - \bar{k}^{(2)}}\|_{\mathbf{L}^\infty} + t\|{f}\|_{\mathbf{L}^\infty} \sup_{\tau\in[0,t]}\|{\phi^{(1)}(\tau,\cdot)-\phi^{(2)}(\tau,\cdot)}\|_{\mathbf{L}^\infty} \\
    &+ C(1+t) \int_0^t \|{k^{(1)}(\tau,\cdot)-k^{(2)}(\tau,\cdot)}\|_{\mathbf{L}^\infty} \d\tau,
\end{align*}
where the constant $C=C\left(\norm{\frac{\d}{\d x}\bar{k}^{(1)}}_{\mathbf{L}^\infty}, \sup\nolimits_{\tau\in[0,t]} \norm{\partial_x\phi^{(1)}(\tau,\cdot)}_{\mathbf{L}^\infty}, \|{f}\|_{\mathbf{L}^\infty}, \norm{V'}_{\mathbf{L}^\infty} \right)$.
By Gronwall's inequality, we deduce the desired stability estimate \eqref{eq:ch2_stab_Linf}.
The uniqueness of solutions follows directly from the stability estimate.
\end{proof}

Next we give the following $\mathbf{L}^1$ stability estimate for strong solutions, which leads to the existence, uniqueness, and stability of weak solutions to problem \eqref{eq:bathtub_model}-\eqref{eq:initial_condition} with rough data, cf. \cref{thm:wellposedness_weak}.

\begin{lemma}\label{lemm:forward_stability_L1}
Let \cref{assm:1} be fulfilled. %Theorem~\ref{thm:wellposedness_strong},
Suppose that for $i=1,2$: $\bar{k}^{(i)}\in\mathbf{L}^1([0,\infty);[0,\infty))$ is of bounded variation; and $\phi^{(i)}(t,\cdot)\in\mathbf{L}^\infty([0,\infty);[0,\infty))$ is of bounded variation for all $t\in[0,\infty)$ and its total variation is locally bounded on $t\in[0,\infty)$.
Then the respective solutions $k^{(i)},\,i=1,2,$ satisfy
\begin{align}\label{eq:ch2_stab_L1}
    \|{k^{(1)}(t,\cdot)-k^{(2)}(t,\cdot)}\|_{\mathbf{L}^1} \leq C(t) \left( \|{\bar{k}^{(1)}-\bar{k}^{(2)}}\|_{\mathbf{L}^1} +  \int_0^t \|{\phi^{(1)}(\tau,\cdot) -\phi^{(2)}(\tau,\cdot)}\|_{\mathbf{L}^1}\d\tau \right),
\end{align}
for $t\in[0,\infty)$, where $C(t)$ depends on $t,\norm{f}_{\mathbf{L}^\infty},\norm{V'}_{\mathbf{L}^\infty},\norm{w}_{\mathbf{L}^\infty},\mathrm{TV}\left(\bar{k}^{(1)}\right)$, and \\ $ \sup\nolimits_{\tau\in[0,t]} \mathrm{TV}\left(\phi^{(1)}(\tau,\cdot)\right)$.
\end{lemma}

\begin{proof}
The starting point of the proof is the solution representation \eqref{eq:ch2_stab_calc_1}-\eqref{eq:ch2_stab_calc_2}.
Indeed, using the identity \eqref{eq:ch2_stab_calc_1}, we obtain 
\begin{align*}
    &\|{k^{(1)}(t,\cdot)-k^{(2)}(t,\cdot)}\|_{\mathbf{L}^1}\\
    \leq& \|{\bar{k}^{(1)}-\bar{k}^{(2)}}\|_{\mathbf{L}^1}
    + \int_0^\infty \left| \bar{k}^{(1)}\left(x+\xi^{(1)}(t)\right)-\bar{k}^{(1)}\left(x+\xi^{(2)}(t)\right) \right| \d x \\
    &+ \|{f}\|_{\mathbf{L}^\infty} \int_0^t \|{\phi^{(1)}(\tau,\cdot)-\phi^{(2)}(\tau,\cdot)}\|_{\mathbf{L}^1} \d\tau  \\
    &  + \|{f}\|_{\mathbf{L}^\infty} \int_0^t \int_0^\infty \left| \phi^{(1)}\left(\tau,x+\xi^{(1)}(t)-\xi^{(1)}(\tau)\right)\! -\! \phi^{(1)}\left(\tau,x+\xi^{(2)}(t)-\xi^{(2)}(\tau)\right) \right| \d x \d\tau \\
    \leq& \|{\bar{k}^{(1)}-\bar{k}^{(2)}}\|_{\mathbf{L}^1} + \mathrm{TV}\left(\bar{k}^{(1)}\right) E(t,0) + \|{f}\|_{\mathbf{L}^\infty} \int_0^t \|{\phi^{(1)}(\tau,\cdot)-\phi^{(2)}(\tau,\cdot)}\|_{\mathbf{L}^1} \d\tau \\
    &+ \norm{f}_{\mathbf{L}^\infty} \int_0^t \mathrm{TV}\left(\phi^{(1)}(\tau,\cdot)\right) E(t,\tau) \d\tau,
\end{align*}
where we have used the inequality $\int_0^\infty |u(x+s)-u(x)|\d x\leq s\cdot \mathrm{TV}(u)$ for any function $u=u(x)$ of bounded variation \cite[{pp. 267-269}]{brezis2011functional},
with $E(t,\tau)$ defined in \eqref{eq:E_t_tau}.
Then it follows from \eqref{eq:ch2_stab_calc_2} that
\begin{equation*}
    E(t,\tau) \leq \norm{V'}_{\mathbf{L}^\infty} \|{w}\|_{\mathbf{L}^\infty} \int_\tau^t \|{k^{(1)}(\tilde{\tau},\cdot)-k^{(2)}(\tilde{\tau},\cdot)}\|_{\mathbf{L}^1} \d\tilde{\tau}.
\end{equation*}
From the above two inequalities, we obtain
\begin{align*}
    \|{k^{(1)}(t,\cdot)-k^{(2)}(t,\cdot)}\|_{\mathbf{L}^1}
    \leq& \|{\bar{k}^{(1)}-\bar{k}^{(2)}}\|_{\mathbf{L}^1} + \|{f}\|_{\mathbf{L}^\infty} \int_0^t \|{\phi^{(1)}(\tau,\cdot)-\phi^{(2)}(\tau,\cdot)}\|_{\mathbf{L}^1}\d\tau \\
    &+ C(1+t) \int_0^t \|{k^{(1)}(\tau,\cdot)-k^{(2)}(\tau,\cdot)}\|_{\mathbf{L}^1} \d\tau,
\end{align*}
with the constant $C=C(\mathrm{TV}\left(\bar{k}^{(1)}\right), \sup_{\tau\in[0,t]} \mathrm{TV}\left(\phi^{(1)}(\tau,\cdot)\right), \|{f}\|_{\mathbf{L}^\infty}, \norm{V'}_{\mathbf{L}^\infty}, \|{w}\|_{\mathbf{L}^\infty})$.
By Gronwall's inequality, we deduce the stability estimate \eqref{eq:ch2_stab_L1}.
\end{proof}

Using \cref{lemm:forward_stability_L1}, we can now derive the existence, uniqueness, and stability of weak solutions to problem \eqref{eq:bathtub_model}--\eqref{eq:initial_condition}, where the weak solutions are in the following sense.
\begin{definition}\label{def:weak}
A function $k\in \mathbf{L}^1([0,\infty )\times [0,\infty);[0,\infty))$ is a weak solution to problem \eqref{eq:bathtub_model}--\eqref{eq:initial_condition} if
\begin{align*}
\int_0^\infty\int_0^\infty k (t,x) &\partial_t\varphi (t,x) - k (t,x) v(t) \partial_x \varphi(t,x) + f(t) \phi(t,x) \varphi(t,x) \d x\d t\\&+\int_0^\infty \bar k(x)\varphi (0,x)\d x= 0, 
\quad\mathrm{with} \quad v(t)=V \left( \int_0^\infty k(t,y) w(y) \d y \right),
\end{align*}
holds for all $\varphi \in \mathbf{C}^1_{\mathrm{c}}([0,\infty )\times(0,\infty))$.
\end{definition}

\begin{theorem}\label{thm:wellposedness_weak}
    Suppose that the initial data $\bar{k}\in\mathbf{L}^1([0,\infty);[0,\infty))$ is of bounded variation, that $\phi(t,\cdot)\in\mathbf{L}^1([0,\infty);[0,\infty))$ is of bounded variation for all $t\in[0,\infty)$ and its total variation is locally bounded on $t\in[0,\infty)$, and that \cref{assm:1} holds.
    Then there exists a unique weak solution $k\in\mathbf{L}^1([0,\infty)\times[0,\infty);[0,\infty))$ that satisfies  problem \eqref{eq:bathtub_model}-\eqref{eq:initial_condition} in the sense of \cref{def:weak}.
    Moreover, the stability estimate \eqref{eq:ch2_stab_L1} still holds.
\end{theorem}

Finally, we give stability estimates on the strong and weak solutions. These estimates can be easily obtained from \eqref{eq:char_xi}-\eqref{eq:char_sol}. Hence, we will omit the proofs.

\begin{proposition}
The following statements hold.\label{prop:estimate}
\begin{itemize}
\item[{\rm(i)}]    Under the conditions of \cref{thm:wellposedness_strong}, the strong solution $k$ to  problem \eqref{eq:bathtub_model}-\eqref{eq:initial_condition} satisfies
    \begin{align*}
        \|{k(t,\cdot)}\|_{\mathbf{L}^\infty} \leq& \|{\bar{k}}\|_{\mathbf{L}^\infty} + \sup\nolimits_{\tau\in[0,t]}\|{\phi}(\tau,\cdot)\|_{\mathbf{L}^\infty} \int_0^t f(\tau)\d\tau, \\
        \|{\partial_xk(t,\cdot)}\|_{\mathbf{L}^\infty} \leq& \|{\bar{k}'}\|_{\mathbf{L}^\infty} + \sup\nolimits_{\tau\in[0,t]}\|{\partial_x\phi(\tau,\cdot)}\|_{\mathbf{L}^\infty} \int_0^t f(\tau)\d\tau.
    \end{align*}
\item[{\rm(ii)}]
    Under the conditions of \cref{thm:wellposedness_weak}, the weak solution $k$ to problem \eqref{eq:bathtub_model}-\eqref{eq:initial_condition} satisfies
    \begin{align*}
        \norm{k(t,\cdot)}_{\mathbf{L}^1} \leq& \norm{\bar{k}}_{\mathbf{L}^1} + \sup\nolimits_{\tau \in [0,t]}\norm{\phi(\tau,\cdot)}_{\mathbf{L}^1} \int_0^t f(\tau)\d\tau, \\
        \mathrm{TV}\left(k(t,\cdot)\right) \leq& \mathrm{TV}\left(\bar{k}\right) + \sup\nolimits_{\tau \in [0,t]} \mathrm{TV} \left(\phi(\tau,\cdot)\right) \int_0^t f(\tau)\d\tau.
    \end{align*}

\end{itemize}
\end{proposition}

\section{Inverse problem of recovering inflow rates}
\label{sec:inverse1}
In this section, we discuss an inverse source problem associated with \eqref{eq:bathtub_model}-\eqref{eq:initial_condition}.
Our goal is to recover the inflow rates $f(t)$ for $t\in[0,T]$, where $[0,T]$ is a predefined time horizon, from the lateral observation $k(t,0)$ at $x=0$ for $t\in[0,T]$. This observation represents the density of exiting vehicles. Throughout this section, we assume that the nonlocal kernel $w$, the velocity function $V$, the initial data $\bar{k}$, and the inflow distribution $\phi$ are known and satisfy the following conditions. Moreover, we use $\norm{\,\cdot\,}_{\mathbf{L}^p}$ for $p\in[1,\infty]$ as shorthand for the $\mathbf{L}^p$ norm on $[0,T]$, $[0,\infty)$, or $[0,T]\times[0,\infty)$, depending on the function's domain. A similar convention applies to $\norm{\,\cdot\,}_{\mathbf{W}^{1,p}}$ for $p\in[1,\infty]$.

\begin{assumption}\label{assm:2}
{\rm(i)} The nonlocal kernel $w$ is a constant function, i.e., $ w(y)=\frac1L\mathbf{1}_{[0,L]}(y)$, for a given constant $L>0$; {\rm(ii)} the velocity function $V$ is bounded and Lipschitz continuous on $[0,\infty)$ and satisfies $V(k)\geq0$ for all $k\geq0$; {\rm(iii)} the initial data $\bar{k}\in\mathbf{L}^\infty([0,\infty);[0,\infty))$ is $\mathbf{C}^2$ smooth and supported on $[0,L]$; and {\rm(iv)} the inflow distribution $\phi\in\mathbf{L}^\infty([0,T]\times[0,\infty);[0,\infty))$ is $\mathbf{C}^2$ smooth, supported on $[0,T]\times[0,L]$, and satisfies $\int_0^\infty \phi(t,x)\d x=1$ and $\phi(t,0)>0$ for all $t\in[0,T]$. 
\end{assumption}

We shall also assume that the exact $f$ belongs to $\mathbf{L}^\infty([0,T];[0,\infty))$. Then we have the following result for the direct problem, which is a direct consequence of \cref{thm:wellposedness_strong} and \cref{prop:estimate}.
\begin{proposition}\label{prop:ch3_forward_wellposedness}
Suppose that $f\in\mathbf{L}^\infty ([0,T];[0,\infty))$ and \cref{assm:2} holds. Then problem \eqref{eq:bathtub_model}-\eqref{eq:initial_condition} has a unique strong solution and the solution satisfies:
\begin{align*}
    \sup\nolimits_{t\in[0,T]}\norm{k(t,\cdot)}_{\mathbf{L}^1} \leq \norm{\bar{k}}_{\mathbf{L}^1} + T \norm{f}_{\mathbf{L}^\infty} \;\;\; \mathrm{and} \;\;\;
    \norm{k}_{\mathbf{W}^{1,\infty}} \leq C \left( \norm{\bar{k}}_{\mathbf{W}^{1,\infty}} + \norm{\phi}_{\mathbf{W}^{1,\infty}} \right),
\end{align*}
where the constant $C$ depends on $T,\norm{f}_{\mathbf{L}^\infty}$, and $\norm{V}_{\mathbf{L}^\infty}$.
\end{proposition}

\subsection{Reformulation as an integral equation}
In this subsection, we develop the key tool for analyzing the concerned inverse source problem. We define the following quantity:
\[ \delta(t)\doteq\int_0^\infty k(t,x) \d x, \quad t\in[0,T]. \]
Physically, $\delta(t)$ represents the total number of vehicles at time $t$.
Directly from the model \eqref{eq:bathtub_model}, we derive that the solution $k$ is supported on $[0,T]\times[0,L]$ and obtain the following nonlinear ODE for $\delta$:
\begin{align}\label{eq:inv_source_1}
    \frac{\d\delta(t)}{\d t} = f(t) - V \left( \frac{\delta(t)}{L} \right) k(t,0).
\end{align}
Given $\delta(t)$, setting $x=0$ in the solution representation formula \eqref{eq:char_sol} yields
\begin{align}\label{eq:inv_source_2}
    k(t,0) - \bar{k}(\xi(t)) = \int_0^t f(\tau) \phi(\tau, \xi(t)-\xi(\tau)) \d\tau, \quad t\in[0,T],
\end{align}
with
\begin{align}\label{eq:inv_source_3}
    \xi(t) = \int_0^t V \left( \frac{\delta(\tau)}{L} \right) \d\tau, \quad t\in[0,T].
\end{align}
By substituting \eqref{eq:inv_source_1} into \eqref{eq:inv_source_2} (i.e., to eliminate $f$), we obtain the following integro-differential equation: for $t\in [0,T]$
\begin{align}\label{eq:inv_source_4}
    &k(t,0) - \bar{k}(\xi(t)) = \int_0^t \left( \frac{\d\delta(\tau)}{\d\tau} + V \left(\frac{\delta(\tau)}{L}\right) k(\tau,0) \right) \phi(\tau, \xi(t)-\xi(\tau)) \d\tau.
\end{align}
Once we have solved the quantity $\delta(t)$ for $t\in[0,T]$ from \eqref{eq:inv_source_3}-\eqref{eq:inv_source_4}, we can recover $f(t)$ for $t\in[0,T]$ using the defining identity \eqref{eq:inv_source_1}, i.e., differentiating the recovered $\delta$.

Next we reformulate \eqref{eq:inv_source_3}-\eqref{eq:inv_source_4} by defining 
\begin{equation*}
\bar{\delta} \doteq \int_0^\infty \bar{k}(x) \d x\quad \mbox{and}\quad v^{\delta}(t) \doteq V\left(\frac{\delta(t)}{L}\right).
\end{equation*}
By applying integration by parts to \eqref{eq:inv_source_4}, we obtain
\begin{align*}
    k(t,0) - \bar{k}(\xi(t)) =& \delta(t)\phi(t,0) - \bar{\delta}\phi(0,\xi(t)) - \int_0^t \delta(\tau)\partial_t\phi(\tau, \xi(t)-\xi(\tau)) \d\tau \\
    &+ \int_0^t v^{\delta}(\tau) \left( k(\tau,0) \phi (\tau, \xi(t)-\xi(\tau)) + \delta(\tau) \partial_x\phi(\tau, \xi(t)-\xi(\tau)) \right) \d\tau. 
\end{align*}
Consequently,
\begin{align}
    \delta(t) =& \frac1{\phi(t,0)} \left( k(t,0) - \bar{k} \left( \int_0^t v^{\delta}(\tau)\d\tau \right) + \bar{\delta}\phi\left( 0, \int_0^t v^{\delta}(\tau)\d\tau \right) \right) \nonumber\\
    & - \frac1{\phi(t,0)} \int_0^t v^{\delta}(\tau) k(\tau,0) \phi \left( \tau, \int_\tau^t v^{\delta}(\tilde{\tau})\d\tilde{\tau} \right) + v^{\delta}(\tau) \delta(\tau) \partial_x\phi \left( \tau, \int_\tau^t v^{\delta}(\tilde{\tau})\d\tilde{\tau} \right) \d\tau  \notag \\
    & + \frac1{\phi(t,0)} \int_0^t \delta(\tau) \partial_t\phi \left( \tau, \int_\tau^t v^{\delta}(\tilde{\tau})\d\tilde{\tau} \right) \d\tau. \label{eq:ch3_volterra}
\end{align}
This is a nonlinear Volterra integral equation of the second kind for the unknown $\delta$, and it lays the foundation for the analysis of the considered inverse source problem.

\begin{remark}
The above derivation relies on the specific choice of the nonlocal kernel $w(y)=\frac1L\mathbf{1}_{[0,L]}(y)$ and the condition that both $\bar{k}$ and $\phi$ vanish for $x>L$. Here, $L$ can be interpreted as the maximum trip distance, and the direct problem is indeed solved on the finite domain $[0,T]\times[0,L]$. Only in this case the velocity field $v=V(\int_0^\infty k(t,y)w(y)\d y)=V(\delta(t)/L)$ is a function of $\delta(t)$, enabling us to write down the ODE \eqref{eq:inv_source_1} and reformulate the inverse problem as a Volterra integral equation \eqref{eq:ch3_volterra}, and moreover,
equations \eqref{eq:inv_source_1}-\eqref{eq:inv_source_3} give rise to the following integro-differential equation for $\delta$:
\begin{equation*}
    \frac{\d\delta(t)}{\d t} = f(t) - V \left( \frac{\delta(t)}{L} \right) \left( \bar{k}\left( \int_0^t V \left( \frac{\delta(\tau)}{L} \right) \d\tau \right) + \int_0^t f(\tau) \phi\left(\tau, \int_\tau^t V \left( \frac{\delta(\tilde{\tau})}{L} \right) \d\tilde{\tau} \right) \d \tau \right) ,
\end{equation*}
which is also referred to as the trip-based model using the network density-speed relation; see the discussions in \cite{laval2023effect}.
\end{remark}

\subsection{Well-posedness of the inverse problem}\label{sec:math_analysis}
Now we establish the existence, uniqueness, and stability of solutions to the inverse problem that is reformulated as the Volterra integral equation \eqref{eq:ch3_volterra}. The following theorem presents the main result of this section.
The analysis employs the method of successive approximations, a standard tool for analyzing Volterra integral equations of the second kind \cite{Lamm:2000}. In the theorem, $\phi_{\min,0}\doteq\min_{t\in[0,T]}\phi(t,0)$ is a positive constant.
\begin{theorem}\label{thm:inverse_source}
Let \cref{assm:2} be fulfilled. Then for any observational data $k(\cdot,0)\in\mathbf{W}^{1,\infty}([0,T])$, there exists a unique $\delta\in\mathbf{W}^{1,\infty}([0,T])$ satisfying the Volterra integral equation \eqref{eq:ch3_volterra} and a unique $f\in\mathbf{L}^\infty([0,T])$ given by \eqref{eq:inv_source_1}. Moreover, we have the following stability estimates:
\begin{align}
    \|{\delta - \tilde{\delta}}\|_{\mathbf{L}^\infty} \leq& C_\delta \|{k(\cdot,0)-\tilde{k}(\cdot,0)}\|_{\mathbf{L}^\infty}, \label{eq:formal_stab_estm_delta}\\
    \|{f - \tilde{f}}\|_{\mathbf{L}^\infty} \leq& C_f \|{k(\cdot,0)-\tilde{k}(\cdot,0)}\|_{\mathbf{W}^{1,\infty}}, \label{eq:formal_stab_estm_f}
\end{align}
where the constant $C_\delta$ depends on $T,L,V,\bar{\delta},\phi_{\mathrm{min},0}^{-1},\|{k(\cdot,0)}\|_{\mathbf{L}^\infty},\|{\phi}\|_{\mathbf{W}^{2,\infty}},\|{\bar{k}}\|_{\mathbf{W}^{1,\infty}}$, and the constant $C_f$ depends on $T,L,V,\bar{\delta},\phi_{\mathrm{min},0}^{-1},\|{k(\cdot,0)}\|_{\mathbf{L}^\infty},\|{\phi}\|_{\mathbf{W}^{2,\infty}},\|{\bar{k}}\|_{\mathbf{W}^{2,\infty}}$.
\end{theorem}
\begin{proof}
We divide the lengthy proof into four steps.\\
\textbf{(Step 1)} We show the existence of solutions. Let $\{\delta_n\}_{n=0}^\infty$ be the sequence produced by the successive approximations:
\begin{align}
    \delta_0(t) =& 0, \notag \\
    \delta_{n+1}(t) =& \frac1{\phi(t,0)} \left( k(t,0) - \bar{k} \left( \int_0^t v^{\delta_n}(\tau)\d\tau \right) + \bar{\delta}\phi\left( 0, \int_0^t v^{\delta_n}(\tau)\d\tau \right) \right) \notag \\
    & - \frac1{\phi(t,0)} \int_0^t v^{\delta_n}(\tau) k(\tau,0) \phi \left( \tau, \int_\tau^t v^{\delta_n}(\tilde{\tau})\d\tilde{\tau} \right)\notag\\
    &+v^{\delta_n}(\tau) \delta_n(\tau) \partial_x\phi \left( \tau, \int_\tau^t v^{\delta_n}(\tilde{\tau})\d\tilde{\tau} \right) \d\tau \notag \\
    & + \frac1{\phi(t,0)} \int_0^t \delta_n(\tau) \partial_t\phi \left( \tau, \int_\tau^t v^{\delta_n}(\tilde{\tau})\d\tilde{\tau} \right) \d\tau, \quad n\geq0. \label{eq:ch3_successive}
\end{align}
We first give a uniform $\mathbf{L}^\infty$ bound on the sequence $\{\delta_n\}_{n=0}^\infty$.
From the identity \eqref{eq:ch3_successive}, we deduce
\begin{align*}
    |\delta_{n+1}(t)| \leq \frac1{\phi_{\mathrm{min},0}} \left( C_1 + C_2\int_0^t |\delta_n(\tau)| \d\tau \right), \quad \quad t\in[0,T],\,n\geq0,
\end{align*}
where the constants $C_1=\norm{k(\cdot,0)}_{\mathbf{L}^\infty} + \norm{\bar{k}}_{\mathbf{L}^\infty} + \bar{\delta}\norm{\phi}_{\mathbf{L}^\infty} + T\norm{V}_{\mathbf{L}^\infty}\norm{k(\cdot,0)}_{\mathbf{L}^\infty}\norm{\phi}_{\mathbf{L}^\infty}$ and $C_2=\norm{V}_{\mathbf{L}^\infty}\norm{\partial_x\phi}_{\mathbf{L}^\infty} + \norm{\partial_t\phi}_{\mathbf{L}^\infty}$.
Let $W(t)=\frac{C_1}{\phi_{\mathrm{min},0}} e^{\frac{C_2}{\phi_{\mathrm{min},0}}t}$, which satisfies the following integral equation for $t\in[0,T]$:
\begin{align*}
    W(t) = \frac1{\phi_{\mathrm{min},0}} \left( C_1 + C_2\int_0^t W(\tau) \d\tau \right).
\end{align*}
Then an induction argument gives $|\delta_n(t)|\leq W(t)$ for all $n\geq0$ and $t\in[0,T]$.
This leads to the following uniform $\mathbf{L}^\infty$ bound on the sequence $\{\delta_n\}_{n=0}^\infty$: 
\begin{align}\label{eq:ch3_delta_Linf_bound}
    \sup\nolimits_{n\geq0} \norm{\delta_n}_{\mathbf{L}^\infty} \leq C_3 = C_3\left( T,V,\bar{\delta},\phi_{\mathrm{min},0}^{-1},\norm{k(\cdot,0)}_{\mathbf{L}^\infty},\norm{\bar{k}}_{\mathbf{L}^\infty},\norm{\phi}_{\mathbf{W}^{1,\infty}} \right).
\end{align}
\textbf{(Step 2)} We bound the difference $\delta_{n+1} - \delta_n$. By \eqref{eq:ch3_successive}, for $t\in[0,T]$ and $n\geq1$, direction computation gives:
\begin{align*}
    &|\delta_{n+1}(t) - \delta_n(t)| \\
    \leq& \frac{1}{\phi(t,0)} \left[  \left( \norm{\bar{k}'}_{\mathbf{L}^\infty} + \bar{\delta}\norm{\partial_x\phi}_{\mathbf{L}^\infty} + \norm{k(\cdot,0)}_{\mathbf{L}^\infty}\norm{\phi}_{\mathbf{L}^\infty} \right) \int_0^t \left| v^{\delta_n}(\tau) - v^{\delta_{n-1}}(\tau) \right| \d\tau \right. \\
    & + \left(  \left( \norm{V}_{\mathbf{L}^\infty} + \frac{\norm{V'}_{\mathbf{L}^\infty}}{L}\norm{\delta_n}_{\mathbf{L}^\infty} \right) \norm{\partial_x\phi}_{\mathbf{L}^\infty} + \norm{\partial_t\phi}_{\mathbf{L}^\infty} \right) \int_0^t \left| \delta_n(\tau) - \delta_{n-1}(\tau) \right| \d\tau \\
    & + \norm{V}_{\mathbf{L}^\infty}\norm{k(\cdot,0)}_{\mathbf{L}^\infty}\norm{\partial_x\phi}_{\mathbf{L}^\infty} \int_0^t\int_\tau^t \left| v^{\delta_n}(\tilde{\tau}) - v^{\delta_{n-1}}(\tilde{\tau}) \right| \d\tilde{\tau}\d\tau \\
    & + \left. \norm{\delta_n}_{\mathbf{L}^\infty} \left(  \norm{V}_{\mathbf{L}^\infty} \norm{\partial^2_{xx}\phi}_{\mathbf{L}^\infty} + \norm{\partial^2_{tx}\phi}_{\mathbf{L}^\infty} \right) \int_0^t\int_\tau^t \left| v^{\delta_n}(\tilde{\tau}) - v^{\delta_{n-1}}(\tilde{\tau}) \right| \d\tilde{\tau}\d\tau \right] \\
    \leq& C_4 \left( \int_0^t (1+\tau) \left| v^{\delta_n}(\tau) - v^{\delta_{n-1}}(\tau) \right| \d\tau + \int_0^t |\delta_n(\tau) - \delta_{n-1}(\tau)| \d\tau \right) \\
    \leq& C_4\left((1+T)\frac{\norm{V'}_{\mathbf{L}^\infty}}{L}+1\right)  \int_0^t\left| \delta_n(\tau) -\delta_{n-1}(\tau) \right| \d\tau 
   = C_5 \int_0^t |\delta_n(\tau) - \delta_{n-1}(\tau)| \d\tau,
\end{align*}
where the constant $C_5=C_5\left(T,L,V,\bar{\delta},\phi_{\mathrm{min,0}}^{-1},\norm{k(\cdot,0)}_{\mathbf{L}^\infty},\norm{\bar{k}}_{\mathbf{W}^{1,\infty}},\norm{\phi}_{\mathbf{W}^{2,\infty}}\right)$, and the estimate \eqref{eq:ch3_delta_Linf_bound} has been used.
Moreover, the estimate \eqref{eq:ch3_delta_Linf_bound} gives: 
\begin{equation*}
    |\delta_1(t) - \delta_0(t)| \leq \norm{\delta_1}_{\mathbf{L}^\infty} \leq C_3,\quad t\in[0,T].
\end{equation*}
Then by mathematical induction, we deduce
\begin{align}
    |\delta_{n+1}(t) - \delta_n(t)| \leq C_3C_5^n\frac{t^n}{n!} , \quad t\in[0,T],\,n\geq0.
\end{align}
Therefore, the infinite series $\sum_{n=0}^\infty (\delta_{n+1}(t)-\delta_{n}(t))$ converge uniformly for $t\in[0,T]$.
Thus, the sequence $\{\delta_n\}_{n=0}^\infty$ converges uniformly to some $\delta\in\mathbf{L}^\infty([0,T])$. Furthermore, it is direct to verify that the limit $\delta$ satisfies the integral equation \eqref{eq:ch3_volterra} pointwise.\\
\textbf{(Step 3)} We obtain the recovery formula of $f$ from $\delta$. Indeed, by differentiating the identity \eqref{eq:ch3_volterra}, we get:
\begin{align}
    \frac{\d\delta(t)}{\d t} =& - v^\delta(t) k(t,0) - \frac{v^\delta(t)\delta(t) \partial_x\phi(t,0)}{\phi(t,0)}\nonumber\\
     &- \frac{v^\delta(t)}{\phi(t,0)} \int_0^t v^{\delta}(\tau) k(\tau,0) \partial_x\phi \left( \tau, \int_\tau^t v^{\delta}(\tilde{\tau})\d\tilde{\tau} \right) \d\tau \notag \\
    & + \frac1{\phi(t,0)} \left( \partial_tk(t,0) - \bar{k}' \left( \int_0^t v^{\delta}(\tau)\d\tau \right) v^\delta(t) + \bar{\delta}\partial_x\phi\left( 0, \int_0^t v^{\delta}(\tau)\d\tau \right) v^\delta(t) \right) \notag \\
    & - \frac{v^\delta(t)}{\phi(t,0)} \int_0^t v^{\delta}(\tau) \delta(\tau) \partial_{xx}^2\phi \left( \tau, \int_\tau^t v^{\delta}(\tilde{\tau})\d\tilde{\tau} \right) - \delta(\tau) \partial_{tx}^2\phi \left( \tau, \int_\tau^t v^{\delta}(\tilde{\tau})\d\tilde{\tau} \right) \d\tau . \label{eq:ch3_delta_dot}
\end{align}
This identity implies that $\delta\in\mathbf{W}^{1,\infty}([0,T])$,
and we can recover the inflow rates $f(t)$ for $t\in[0,T]$ from the recovered $\delta(t)$ and the measurement $k(t,0)$ for $t\in[0,T]$ using \eqref{eq:inv_source_1} and \eqref{eq:ch3_delta_dot}.\\
\textbf{(Step 4)} We derive the stability estimate for the inverse problem. Suppose that $\delta$ and $\tilde{\delta}$ are solutions to equation \eqref{eq:ch3_volterra} corresponding to the observational data $k(\cdot,0)$ and $\tilde{k}(\cdot,0)$, respectively. A similar argument to the one used in the above existence proof gives:
\begin{align*}
    |\delta(t) - \tilde{\delta}(t)| \leq &\frac1{\phi_{\mathrm{min},0}} \left( |k(t,0) - \tilde{k}(t,0)| + \norm{\phi}_{\mathbf{L}^\infty}\norm {V}_{\mathbf{L}^\infty} \int_0^t |k(\tau,0) - \tilde{k}(\tau,0)| \d\tau \right)\\
     &+ C_5 \int_0^t |\delta(\tau) - \tilde{\delta}(\tau)| \d\tau,
\end{align*}
for $t\in[0,T]$, where the constant $C_5$ is as in the existence proof.
Using Gronwall's inequality, we deduce
\begin{align*}
    \|{\delta - \tilde{\delta}}\|_{\mathbf{L}^\infty} \leq C_6 \|{k(\cdot,0)-\tilde{k}(\cdot,0)}\|_{\mathbf{L}^\infty},
\end{align*}
where the constant $C_6=C_6\left( T,L,V,\bar{\delta},\phi_{\mathrm{min},0}^{-1},\|{k(\cdot,0)}\|_{\mathbf{L}^\infty},\|{\bar{k}}\|_{\mathbf{W}^{1,\infty}},\|{\phi}\|_{\mathbf{W}^{2,\infty}} \right)$.
Suppose that $f$ and $\tilde{f}$ are the recovered inflow rates from $\delta,k(\cdot,0)$ and $\tilde{\delta}, \tilde{k}(\cdot,0)$, respectively, using \eqref{eq:inv_source_1} and \eqref{eq:ch3_delta_dot}.
Using integration by parts, we can rewrite \eqref{eq:ch3_delta_dot} as:
\begin{align}
    \frac{\d\delta(t)}{\d t} =& - v^\delta(t) k(t,0) + \frac1{\phi(t,0)} \left( \partial_tk(t,0) - \bar{k}' \left( \int_0^t v^{\delta}(\tau)\d\tau \right) v^\delta(t)  \right) \notag \\
    & - \frac{v^\delta(t)}{\phi(t,0)} \int_0^t \left( \frac{\d\delta(\tau)}{\d\tau} + v^{\delta}(\tau) k(\tau,0) \right) \partial_x\phi \left( \tau, \int_\tau^t v^{\delta}(\tilde{\tau})\d\tilde{\tau} \right) \d\tau. \label{eq:ch3_delta_dot_rewrite}
\end{align}
We can bound $\|{\frac{\d\delta}{\d t}-\frac{\d\tilde{\delta}}{\d t}}\|_{\mathbf{L}^\infty}$ using the above identity and Gronwall's inequality, and then derive an estimate on $\|{f-\tilde{f}}\|_{\mathbf{L}^\infty}$ from \eqref{eq:inv_source_1}. Finally we have
\begin{align}\label{eq:ch3_stablity_estimate_f}
    \|{f - \tilde{f}}\|_{\mathbf{L}^\infty} \leq C_7 \|{k(\cdot,0)-\tilde{k}(\cdot,0)}\|_{\mathbf{W}^{1,\infty}},
\end{align}
where the constant $C_7=C_7\left( T,L,V,\bar{\delta},\phi_{\mathrm{min},0}^{-1},\|{k(\cdot,0)}\|_{\mathbf{L}^\infty},\|{\bar{k}}\|_{\mathbf{W}^{2,\infty}},\|{\phi}\|_{\mathbf{W}^{2,\infty}}\right)$.
The uniqueness of solutions follows from the stability estimate \eqref{eq:ch3_stablity_estimate_f}. This completes the proof of the theorem.
\end{proof}

\begin{remark}
The analysis for \cref{thm:inverse_source} requires the inflow distribution $\phi$ to be $\mathbf{C}^2$ smooth. Nonetheless, it is sufficient to have the smoothness condition within the domain $[0,T]\times[0,\norm{V}_{\mathbf{L}^\infty} T]$, which is the domain of dependence for the observed data $k(t,0)$ for $t\in[0,T]$. A practical example is $\phi(t,x)=\frac1L\mathbf{1}_{[0,L]}(x)$, indicating a uniform distribution over $[0,L]$ for all times. In this scenario, the well-posedness of the strong solution for $t\in[0,T]$ and the conclusion of \cref{thm:inverse_source} still applies, provided that $\norm{V}_{\mathbf{L}^\infty}T<L$. Similarly, the smoothness condition for the initial data $\bar{k}$ can be relaxed, requiring $\bar{k}$ to be $\mathbf{C}^2$ smooth only within the spatial domain $[0,\norm{V}_{\mathbf{L}^\infty} T]$.
\end{remark}

\begin{remark}
The proof indicates that the existence of solutions to the inverse problem requires only $k(\cdot,0)\in\mathbf{L}^\infty([0,T])$.
But in the stability estimate \eqref{eq:formal_stab_estm_f}, we have used the $\mathbf{W}^{1,\infty}$ norm of $k(\cdot,0) - \tilde{k}(\cdot,0)$ to control the $\mathbf{L}^\infty$ norm of $f-\tilde{f}$.
The applicability of the $\mathbf{W}^{1,\infty}$ norm in this context is guaranteed by \cref{prop:ch3_forward_wellposedness}.
However, if the objective is to recover the cumulative inflow $F(t) = \int_0^t f(\tau) \d\tau$, then we can get the following stability estimate:
\begin{align*}
    \|{F - \tilde{F}}\|_{\mathbf{L}^\infty} \leq C_F \|{k(\cdot,0)-\tilde{k}(\cdot,0)}\|_{\mathbf{L}^\infty},
\end{align*}
where the constant $C_F$ depends on $ T,L,V,\bar{\delta},\phi_{\mathrm{min},0}^{-1},\|{k(\cdot,0)}\|_{\mathbf{L}^\infty},\|{\phi}\|_{\mathbf{W}^{2,\infty}},\|{\bar{k}}\|_{\mathbf{W}^{1,\infty}} $.
\end{remark}

\subsection{Discussions on nonsmooth inflow distribution}
Lastly, we discuss the case of nonsmooth inflow distribution $\phi$. The discontinuity points will propagate along the characteristic curves, which significantly complicates the analysis of the associated inverse problem. To shed insights, throughout this part, we focus on one specific inflow distribution $\phi$ and make the following assumption in addition to \cref{assm:2}.
\begin{assumption}\label{assm:2_add}
{\rm(i)} The inflow distribution $\phi(t,x)=\frac1L\mathbf{1}_{[0,L]}(x)$ for $t\in[0,T]$ and the discontinuity point of $\phi$ at $x=L$ propagates to the boundary point $x=0$ within the time horizon $[0,T]$; and {\rm(ii)} the velocity function $V$ satisfies $V(k)\geq V_{\mathrm{min}}$ for all $k\geq0$ where $V_{\mathrm{min}}$ is a positive constant. 
\end{assumption}

We keep the notations in \cref{sec:math_analysis}. Under \cref{assm:2_add}, equation \eqref{eq:inv_source_4} becomes
\begin{align}
    &k(t,0) - \bar{k}\left(\int_0^t v^{\delta}(\tau) \d\tau \right) 
    = \frac1L\int_{\eta^\delta(t)}^t \frac{\d\delta(\tau)}{\d\tau} + v^{\delta}(\tau)k(\tau,0) \d\tau, \quad t\in[0,T], \label{eq:inv_source_eta}
\end{align}
with
\begin{align*}
    \eta^\delta(t) \doteq \begin{cases}
        0, \quad \mathrm{if} \quad \int_0^t v^{\delta}(\tau) \d\tau \leq L,\\
        \tilde{t} \quad \mathrm{s.\,t.} \quad \int_{\tilde{t}}^t v^{\delta}(\tau) \d\tau = L, \quad \mathrm{if} \quad \int_0^t v^{\delta}(\tau) \d\tau > L.
    \end{cases}
\end{align*}
The assumption that $V(k)\geq V_{\mathrm{min}}>0$ for all $k\geq0$ guarantees the well-definedness of $\eta^\delta(t)$ for any $t\in[0,T]$. Moreover, $\eta^\delta(t)$ is monotone increasing in $t$, and
\begin{align}\label{eq:eta_step_estm}
    \eta^\delta(t) \leq t - \norm{V}_{\mathbf{L}^\infty}^{-1} L , \quad \mathrm{if} \quad \eta^\delta(t)>0.
\end{align}
We can rewrite problem \eqref{eq:inv_source_eta} as
\begin{align}\label{eq:volterra_eta}
    \delta(t) = \delta\left(\eta^\delta(t)\right) - \int_{\eta^\delta(t)}^t v^\delta(\tau) k(\tau,0) \d\tau + L \left( k(t,0) - \bar{k}\left(\int_0^t v^{\delta}(\tau) \d\tau \right) \right),\quad t\in [0,T].
\end{align}
This forms the basis for the analysis of the existence issue. Note that problem \eqref{eq:volterra_eta} defines $\delta(t)$ for $t\in[0,T]$ in a recursive manner, and condition \eqref{eq:eta_step_estm} ensures that the recursion has no more than $T\norm{V}_{\mathbf{L}^\infty}/L$ steps. 

For the direct problem \eqref{eq:bathtub_model}-\eqref{eq:initial_condition}, \cref{thm:wellposedness_weak} and \cref{prop:estimate} yield the existence of a unique weak solution $k$ with $k(\cdot,0)\in\mathbf{L}^\infty([0,T])$.
Consequently, equation \eqref{eq:inv_source_1} indicates that $\delta\in\mathbf{W}^{1,\infty}([0,T])$. Moreover, we have
\begin{align*}
    \frac{\d\eta^{\delta}(t)}{\d t} = \begin{cases}
        0, \quad &\mathrm{if} \quad \int_0^t v^{\delta}(\tau) \d\tau \leq L, \\
        \frac{v^{\delta}(t)}{v^{\delta}\left(\eta^{\delta}(t)\right)}, \quad &\mathrm{if}  \quad \int_0^t v^{\delta}(\tau) \d\tau > L,
    \end{cases}
\end{align*}
which implies $\eta^\delta\in\mathbf{W}^{1,\infty}([0,T])$, and
\begin{align*}
    \partial_tk(t,0) = \bar{k}'\left(\int_0^t v^{\delta}(\tau) \d\tau \right)v^\delta(t) + f(t) - f\left(\eta^\delta(t)\right) \frac{\d\eta^\delta(t)}{\d t},
\end{align*}
which implies  $k(\cdot,0)\in\mathbf{W}^{1,\infty}([0,T])$.
Once $\delta$ is solved from \eqref{eq:volterra_eta}, the inflow rates $f$ can be recovered from \eqref{eq:inv_source_1}.

To establish the existence of solutions to the integral equation \eqref{eq:volterra_eta}, we consider the following  successive approximations:
 \begin{align*}
     &\delta_0(t) = 0, \\
     &\delta_{n+1}(t) = \begin{cases} \displaystyle
         \bar{\delta} - \int_0^t v^{\delta_n}(\tau) k(\tau,0) \d\tau 
          + L \left( k(t,0) - \bar{k}\left(\int_0^t v^{\delta_n}(\tau) \d\tau \right) \right), \quad \mathrm{if} \quad t \leq t^*_n,\\
       \displaystyle  \delta_{n+1}\left( \eta^{\delta_n}(t) \right)\!\! - \!\!\int_{\eta^{\delta_n}(t)}^t v^{\delta_n}(\tau) k(\tau,0) \d\tau 
          +L \left( k(t,0)\! - \!\bar{k}\left(\int_0^t v^{\delta_n}(\tau) \d\tau \right) \right),\\
          \qquad\qquad \qquad \mathrm{if} \quad t > t^*_n,
     \end{cases}
\end{align*}
where $t^*_n\in[0,\infty)$ is the unique solution to 
$
\int_0^{t^*_n} v^{\delta_n}(\tau) \d\tau = L$.
For any $n\geq0$ and given $\delta_n$, $\delta_{n+1}(t)$ is defined recursively for $t\in[0,T]$, with the recursion having no more than $T\norm{V}_{\mathbf{L}^\infty}/L$ steps. 

The next result gives the existence of a solution, constructed by the above successive approximations.
\begin{proposition}
Let \cref{assm:2} and \cref{assm:2_add} be fulfilled. Then for any observational data 
 $k(\cdot,0)\in\mathbf{W}^{1,\infty}([0,T])$, the series $\sum_{n=0}^\infty (\delta_{n+1}(t)-\delta_n(t))$ converges uniformly for $t\in[0,T]$ to a solution $\delta(t)$ to the integral equation \eqref{eq:volterra_eta}.
\end{proposition}
\begin{proof}
We claim the following estimate:
\begin{align}\label{eq:induc_eta}
     |\delta_{n+1}(t) - \delta_n(t)| \leq  \frac{C^nt^n}{n!}, \quad \mathrm{for} \ t\in[0,T],\,n\geq0,
\end{align}
where the constant $C=C\left(T,L,V,\|{k(\cdot,0)}\|_{\mathbf{L}^\infty}, \|{\bar{k}'}\|_{\mathbf{L}^\infty}\right)$, from which the desired convergence follows directly. 
To show the claim \eqref{eq:induc_eta}, first we give $\mathbf{W}^{1,\infty}$ bounds on the sequences  $\{\eta^{\delta_n}\}_{n=0}^\infty$ and $\{\delta_n\}_{n=0}^\infty$.
For any $n\geq0$, the definition of $\eta^{\delta_n}$ implies $0\leq \eta^{\delta_n}(t) \leq T$ for any $t\in[0,T]$. Moreover, we have $\eta^{\delta_n}(t)\equiv0$ when $t\leq t_n^*$ and 
\begin{align*}
    \int_{\eta^{\delta_n}(t)}^t v^{\delta_n}(\tau)\d\tau = L,
\end{align*}
when $t>t_n^*$. Therefore, $\frac{\d\eta^{\delta_n}(t)}{\d t}$ satisfies
\begin{align*}
    \frac{\d\eta^{\delta_n}(t)}{\d t} = \begin{cases}
        0, \quad &\mathrm{if} \quad t \leq t_n^*,\\
        \frac{v^{\delta_n}(t)}{v^{\delta_n}\left(\eta^{\delta_n}(t)\right)}, \quad &\mathrm{if}  \quad t>t_n^*,
    \end{cases}
\end{align*}
which implies the estimate $\sup_{n\geq0}\norm{\frac{\d\eta^{\delta_n}}{\d t}}_{\mathbf{L}^\infty} \leq \frac{\norm{V}_{\mathbf{L}^\infty}}{V_{\mathrm{min}}}$.
For the sequence $\{\delta_n\}_{n=0}^\infty$, we only need to bound $\norm{\delta_n}_{\mathbf{L}^\infty}$ and $\norm{\frac{\d\delta_n}{\d t}}_{\mathbf{L}^\infty}$ for $n\geq1$ since $\delta_0\equiv0$.
By definition, we have
\begin{align*}
    |\delta_n(t)| \leq \begin{cases}
        \bar{\delta} + C, \quad &\mathrm{if} \quad t \leq t_{n-1}^*, \\
        \left| \delta_n\left(\eta^{\delta_{n-1}}(t)\right) \right| + C, \quad &\mathrm{if} \quad t > t_{n-1}^*,
    \end{cases}
\end{align*}
for $t\in[0,T],\,n\geq1$, where the constant $C=C\left(L,T,V, \norm{k(\cdot,0)}_{\mathbf{L}^\infty}, \norm{\bar{k}}_{\mathbf{L}^\infty}\right)$.
This recursive inequality and the fact that the recursion has no more than $T\norm{V}_{\mathbf{L}^\infty}/L$ steps imply the following estimate:
\begin{align*}
    \sup\nolimits_{n\geq1} \norm{\delta_n}_{\mathbf{L}^\infty} \leq C=C\left(L,T,V,\bar{\delta}, \norm{k(\cdot,0)}_{\mathbf{L}^\infty}, \norm{\bar{k}}_{\mathbf{L}^\infty}\right).
\end{align*}
Furthermore, by definition, we also have
\begin{align*}
    \frac{\d\delta_n(t)}{\d t} = \left\{\begin{aligned}
        \displaystyle &-v^{\delta_n}(t)k(t,0) + L\partial_tk(t,0) - L\bar{k}'\left(\int_0^t v^{\delta_n}(\tau) \,\d\tau \right) v^{\delta_n}(t), \quad \mathrm{if} \quad t \leq t^*_{n-1}, \\
       \displaystyle  &\frac{\d\delta_n}{\d t}\left( \eta^{\delta_{n-1}}(t) \right)\frac{\d\eta^{\delta_{n-1}}(t)}{\d t} - v^{\delta_n}(t)k(t,0) + v^{\delta_n}\left(\eta^{\delta_{n-1}}(t)\right)k\left(\eta^{\delta_{n-1}}(t),0\right) \\
          &\qquad + L\partial_tk(t,0) - L\bar{k}'\left(\int_0^t v^{\delta_n}(\tau) \d\tau \right) v^{\delta_n}(t), \qquad\qquad\quad\ \ \mathrm{if} \quad t > t^*_{n-1}.
    \end{aligned}\right.
\end{align*}
This identity, along with the uniform bound on $\norm{\frac{\d\eta^{\delta_n}}{\d t}}_{\mathbf{L}^\infty}$ for $n\geq0$, implies
\begin{align*}
    \sup\nolimits_{n\geq1} \norm{\frac{\d\delta_n}{\d t}}_{\mathbf{L}^\infty} \leq C=C\left(L,T,V, \bar{\delta},\norm{k(\cdot,0)}_{\mathbf{W}^{1,\infty}}, \norm{\bar{k}}_{\mathbf{W}^{1,\infty}}\right).
\end{align*}
Now we show the claim \eqref{eq:induc_eta} by means of mathematical induction. 
For the base case where $n=0$, the desired estimate \eqref{eq:induc_eta} follows from
\begin{align*}
    |\delta_1(t) - \delta_0(t)| \leq \norm{\delta_1}_{\mathbf{L}^\infty} \leq C \quad \mathrm{for} \quad t\in[0,T].
\end{align*}
Next suppose that \eqref{eq:induc_eta} holds for $n-1$ where $n\geq1$, we show that it holds for $n$. We derive the following estimates for three distinct cases: $t\in[0,t_{n-1}^*],$ $t\in(t_{n-1}^*,t_n^*]$, and $t\in (t_n^*,T]$, where without loss of generality, we have assumed $t^*_{n-1}\leq t^*_n$. The case with $t^*_{n-1}>t^*_n$ can be shown by switching the role of $n-1$ and $n$.
For $t\in[0,t^*_{n-1}]$, similar to the proof of \cref{thm:inverse_source}, we have
\begin{align*}
     |\delta_{n+1}(t) - \delta_n(t)| \leq  \frac{C^nt^n}{n!}.
\end{align*}
Next, for $t\in(t_{n-1}^*,t_n^*]$, from the identity
     \begin{align*}
         \int_0^{t^*_{n-1}} v^{\delta_{n-1}}(\tau) \d\tau = L = \int_0^{t^*_n} v^{\delta_n}(\tau) \d\tau,
     \end{align*}
and the definition $v^\delta(\tau)=V(\delta(\tau)/L)$, we deduce
\begin{align*}
    t^*_n - t^*_{n-1} &\leq \frac{1}{V_{\min}} \int_{t_{n-1}^*}^{t_n^*} v^{\delta_n}(\tau) \d\tau = \frac{1}{V_{\min}}\int_0^{t_{n-1}^*} v^{\delta_{n-1}}(\tau) - v^{\delta_n}(\tau) \d\tau \\
    & \leq \frac{\norm{V'}_{\mathbf{L}^\infty}}{LV_{\mathrm{min}}} \int_0^{t_{n-1}^*}|\delta_{n-1}(\tau)-\delta_n(\tau)|\d\tau \leq \frac{\norm{V'}_{\mathbf{L}^\infty}}{LV_{\mathrm{min}}} \cdot \frac{C^{n}(t^*_{n-1})^n}{n!}.
\end{align*}
We also have
\begin{align*}
    &|\delta_{n+1}(t) - \delta_n(t)|\\
    \leq& |\delta_{n+1}(t) - \delta_{n+1}(t_{n-1}^*)| + |\delta_n(t) - \delta_n(t_{n-1}^*)| + |\delta_{n+1}(t_{n-1}^*) - \delta_n(t_{n-1}^*)| \\
    \leq& \left( \norm{\frac{\d\delta_{n+1}}{\d t}}_{\mathbf{L}^\infty} + \norm{\frac{\d\delta_n}{\d t}}_{\mathbf{L}^\infty} \right) (t_n^* - t_{n-1}^*) + |\delta_{n+1}(t_{n-1}^*) - \delta_n(t_{n-1}^*)|\\
    \leq& \frac{C^{n}(t^*_{n-1})^n}{n!}.
\end{align*}
Last, for $t\in(t_n^*,T]$, we have
\begin{align*}
    |\delta_{n+1}(t) - \delta_n(t)| \leq& \left|\delta_{n+1}\left(\eta^{\delta_n}(t)\right) - \delta_n\left(\eta^{\delta_{n-1}}(t)\right)\right| + \frac{C^nt^n}{n!},
\end{align*}
where the second term on the right hand side is obtained similarly using the argument from the proof of \cref{thm:inverse_source}. To bound the first term, we use the induction hypothesis and the identity
\begin{align*}
    \int_{\eta^{\delta_{n-1}}(t)}^{t} v^{\delta_{n-1}}(\tau) \d\tau = L = \int_{\eta^{\delta_n}(t)}^{t} v^{\delta_n}(\tau) \d\tau,
\end{align*}
to give
\begin{align*}
  \left|\eta^{\delta_n}(t) - \eta^{\delta_{n-1}}(t)\right| \leq& \frac{1}{V_{\min}} \left| \int_{\eta^{\delta_{n-1}}(t)}^{\eta^{\delta_n}(t)} v^{\delta_{n-1}}(\tau)\d\tau \right| = \frac{1}{V_{\min}} \left| \int_{\eta^{\delta_n}(t)}^t v^{\delta_n}(\tau) - v^{\delta_{n-1}}(\tau) \d\tau \right| \\
  \leq& \frac{\norm{V'}_{\mathbf{L}^\infty}}{LV_{\mathrm{min}}}\int_{\eta^{\delta_n}(t)}^t |\delta_n(\tau)-\delta_{n-1}(\tau)|\d\tau 
   \leq \frac{C^nt^n}{n!}.
\end{align*}
Then by the triangle inequality, we obtain
\begin{align*}
    &\left|\delta_{n+1}\left(\eta^{\delta_n}(t)\right) - \delta_n\left(\eta^{\delta_{n-1}}(t)\right)\right| \\
    \leq& \left|\delta_{n+1}\left(\eta^{\delta_{n-1}}(t)\right) - \delta_n\left(\eta^{\delta_{n-1}}(t)\right)\right| + \norm{\frac{\d\delta_{n+1}}{\d t}}_{\mathbf{L}^\infty} \left|\eta^{\delta_n}(t) - \eta^{\delta_{n-1}}(t)\right| \\
    \leq& \left|\delta_{n+1}\left(\eta^{\delta_{n-1}}(t)\right) - \delta_n\left(\eta^{\delta_{n-1}}(t)\right)\right| + \frac{C^nt^n}{n!},
\end{align*}
which gives
\begin{align*}
    |\delta_{n+1}(t) - \delta_n(t)| \leq \left|\delta_{n+1}\left(\eta^{\delta_{n-1}}(t)\right) - \delta_n\left(\eta^{\delta_{n-1}}(t)\right)\right| + \frac{C^nt^n}{n!}.
\end{align*}
This recursive inequality and the fact that the recursion has no more than $T\norm{V}_{\mathbf{L}^\infty}/L$ steps imply the desired claim \eqref{eq:induc_eta}. This completes the induction step and thus also the proof of the proposition.
\end{proof}

Last, we comment on the stability issue. For any given $k(\cdot,0)$ and $\tilde{k}(\cdot,0)$, let $\delta$ and $\tilde{\delta}$ denote the respective solutions to equation \eqref{eq:volterra_eta}. Using a similar argument to that was used in the above proof, we derive the following recursive inequality: 
\begin{align*}
    |\delta(t) - \tilde{\delta}(t)| \leq \left|\delta\left(\eta^\delta(t)\right) - \tilde{\delta}\left(\eta^{\delta}(t)\right) \right| + C \left( \|{k(\cdot,0)-\tilde{k}(\cdot,0)}\|_{\mathbf{L}^\infty} + \int_0^t |\delta(\tau) - \tilde{\delta}(\tau)| \d\tau  \right),
\end{align*}
for $t\in[0,T]$, which leads to a stability estimate for $\delta$:
\begin{equation}
    \|{\delta - \tilde{\delta}}\|_{\mathbf{L}^\infty} \leq C \|{k(\cdot,0)-\tilde{k}(\cdot,0)}\|_{\mathbf{L}^\infty},
\end{equation}
where the constant $C$ depends on $T,L,V,\bar{\delta},\|{k(\cdot,0)}\|_{\mathbf{W}^{1,\infty}},\|{\bar{k}}\|_{\mathbf{W}^{1,\infty}}$.
In contrast to the stability estimate \eqref{eq:formal_stab_estm_delta} for $\delta$ given in \cref{thm:inverse_source}, the above estimate involves the $\mathbf{W}^{1,\infty}$ norm of $k(\cdot,0)$. Regarding the inflow rates $f$, obtaining a stability estimate akin to \eqref{eq:formal_stab_estm_f} from \cref{thm:inverse_source} would require the $\mathbf{W}^{2,\infty}$ norm of $k(\cdot,0)$, which is not accessible with our current approach. We will explore potential methods to overcome this challenge in future research.

\begin{remark}
The analysis in \cref{sec:math_analysis} for the Volterra integral equation \eqref{eq:ch3_volterra} encounters difficulties with discontinuities in $\phi$, as $\partial_x\phi \left( \tau, \int_\tau^t v^\delta(\tilde{\tau})\d\tilde{\tau} \right) \d\tau$ is interpretable only as a signed measure. For the specific case discussed in this subsection, where $\phi(t,x)=\frac1L\mathbf{1}_{[0,L]}(x)$, $\partial_x\phi \left( \tau, \int_\tau^t v^\delta(\tilde{\tau})\d\tilde{\tau} \right) \d\tau$ represents a Dirac point mass at $\tau=\eta^\delta(t)$, giving rise to the $\delta\left(\eta^\delta(t)\right)$ term in equation \eqref{eq:volterra_eta}. Our approach employs successive approximations that incorporate a careful implicit treatment of this term, enabling us to address the challenge posed by discontinuity.
The exploration of more general discontinuous inflow distributions is left for future research.
\end{remark}

\begin{remark}
Let us also mention some existing literature on Volterra integral equations with strongly singular kernels, which can only be interpreted as measures.
In \cite{staffans1976positive}, the existence and boundedness of solutions were demonstrated for the equation
$ \frac{\d x(t)}{\d t} + \int_0^t g(x(t-\tau)) \d\mu(\tau) = f(t)$,
where $\mu$ is a positive definite measure.
In \cite{marques2016integral}, the authors examined the equation
$ x(t) = f(t) + \int_0^t K(t,x(s),s) \d u(s), $
with $u$ being a function of bounded variation and the integral defined in the Kurzweil-Henstock's sense. They established the existence and stability of solutions under specific conditions on $K$ and $u$, using the theory of generalized ODEs. 
Schwabik in \cite{schwabik1982generalized} extended these results to include Perron integrals, and his later work \cite{schwabik1999linear} addressed Stieltjes integrals.
Additionally, \cite{muftahov2017numeric} investigated numerical solutions to a system of nonlinear integral equations related to the Vintage capital model: %s (VCMs):
\begin{align*}
    \int_{y(t)}^t H(t,\tau,x(\tau))\d\tau = x(t), \qquad
    \int_{y(t)}^t K(t,\tau,x(\tau))\d\tau = f(t).
\end{align*}
\end{remark}

\section{Reconstruction method for inflow rates}\label{sec:recon}
In this section, we present a numerical method for recovering the inflow rates $f(t)$ for $t\in[0,T]$ from the given lateral boundary observation $k(t,0)$ for $t\in[0,T]$. Additionally, we conduct an error analysis of the method, subject to the conditions of \cref{thm:inverse_source}.

\subsection{Numerical method}
The proposed numerical method is inspired by the mathematical analysis given in \cref{sec:math_analysis}. First, we rewrite equation \eqref{eq:ch3_volterra} into
\begin{align}
    \delta(t) =& \frac1{\phi(t,0)} \left( k(t,0) - \bar{k} \left( \int_0^t v^{\delta}(\tau)\d\tau \right) + \bar{\delta}\phi\left( 0, \int_0^t v^{\delta}(\tau)\d\tau \right) \right) \nonumber \\
    & - \frac1{\phi(t,0)} \int_0^t v^{\delta}(\tau) k(\tau,0) \phi \left( \tau, \int_\tau^t v^{\delta}(\tilde{\tau})\d\tilde{\tau} \right)\nonumber \\
    &+ \frac1{\phi(t,0)} \int_0^t \delta(\tau) \frac{\d\phi}{\d\tau} \left( \tau, \int_\tau^t v^{\delta}(\tilde{\tau})\d\tilde{\tau} \right) \d\tau. \label{eq:ch3_volterra_reform}
\end{align}
For the time discretization, we take a uniform time mesh $0 = t_0 < t_1 < \cdots < t_{N-1} < t_N = T$, $ t_n = n\Delta t$,
with the mesh size $\Delta t = T/N$. We then discretize \eqref{eq:ch3_volterra_reform} using \eqref{eq:inv_source_3} and the definition $v^\delta(\tau)=V(\delta(\tau)/L)$ as
\begin{align}
    \xi_n =& \sum_{m=0}^{n-1} V\left(\frac{\delta_m}{L}\right) \Delta t, \quad n=1,\cdots,N;\label{eq:ind_xi} \\
    \delta_n =& \frac1{\phi(t_n,0)} \left( k(t_n,0) - \bar{k}(\xi_n) + \bar{\delta}\phi(0,\xi_n) \right)\nonumber \\
     & - \frac1{\phi(t_n,0)} \sum_{m=0}^{n-1} V\left(\frac{\delta_m}{L}\right) k(t_m,0) \phi(t_m,\xi_n-\xi_m) \Delta t \notag \\
    & +\frac1{\phi(t_n,0)} \sum_{m=0}^{n-1} \delta_m \left( \phi(t_{m+1},\xi_n-\xi_{m+1}) - \phi(t_m,\xi_n-\xi_m) \right), \label{eq:ind_delta}
\end{align}
for $n=1,\cdots,N.$
The scheme \eqref{eq:ind_xi}-\eqref{eq:ind_delta} is equipped with initial conditions $\xi_0=0$ and $\delta_0=\bar{\delta}$. Overall, the scheme is explicit, constructed via employing the left point quadrature to relevant integrals. Then we can solve for $\{\delta_n\}_{n=1}^N$ sequentially. Finally, the reconstructed inflow rates $\{f_n\}_{n=0}^{N-1}$ is given by
\begin{align}
    f_n = \frac{\delta_{n+1} - \delta_n}{\Delta t} + V\left(\frac{\delta_n}{L}\right) k(t_n,0), \quad n=0,\cdots,N-1, \label{eq:ind_f}
\end{align}
by approximating the derivative in \eqref{eq:inv_source_1} with a difference quotient.
Given that the exact inflow rates $f$ belongs to $\mathbf{L}^\infty([0,T])$, the reconstructed $f_n$ should be interpreted as an approximation to $\frac1{\Delta t}\int_{t_n}^{t_{n+1}} f(t)\d t$.

\subsection{Error analysis}
Now we give an error estimate for the scheme \eqref{eq:ind_xi}-\eqref{eq:ind_delta}-\eqref{eq:ind_f} when applied to noisy observation data $k^\sigma(\cdot,0)$, where $\sigma\geq 0$ is the noise level in the $\mathbf{L}^\infty$ norm, i.e., $\sigma = \norm{k(\cdot,0)-k^\sigma(\cdot,0)}_{\mathbf{L}^\infty}$.

We first give some Sobolev estimates for the exact $\delta$, $\xi$ and $k(\cdot,0)$. It follows from equations \eqref{eq:inv_source_1}-\eqref{eq:inv_source_3} that $\norm{\frac{d\xi}{\d t}}_{\mathbf{L}^\infty}\leq\norm{V}_{\mathbf{L}^\infty}$ and
\begin{align*}
    \norm{k(\cdot,0)}_{\mathbf{W}^{1,\infty}} + \norm{\delta}_{\mathbf{W}^{1,\infty}} \leq K = K\left( T,V,\norm{f}_{\mathbf{L}^\infty},\norm{\bar{k}}_{\mathbf{W}^{1,\infty}},\norm{\phi}_{\mathbf{W}^{1,\infty}}\right).
\end{align*}
Next we define the pointwise errors $\{E_n^\delta\}_{0\leq n\leq N}$ and $\{E^\xi_{n,m}\}_{0\leq m\leq n\leq N}$ by
\begin{align*}
    E^\delta_n &= |\delta(t_n) - \delta_n|\quad \mbox{and}\quad
    E^\xi_{n,m} = |\xi(t_n) - \xi(t_m) - \xi_n + \xi_m|.
\end{align*}
Upon comparing \eqref{eq:inv_source_3} with \eqref{eq:ind_xi}, we obtain
\begin{align}
    E^\xi_{n,m} &\leq \frac{\norm{V'}_{\mathbf{L}^\infty}}{L} \sum_{p=m}^{n-1} \int_{t_p}^{t_{p+1}} (K (\tau - t_p) + E^\delta_p)\d\tau \nonumber\\
     & \leq C\Big( (n-m)(\Delta t)^2 + \sum_{p=m}^{n-1} E^\delta_p \Delta t \Big),\label{eq:error_analysis_xi}
\end{align}
where the constant $C=C\left( T,L,V,\norm{f}_{\mathbf{L}^\infty},\norm{\bar{k}}_{\mathbf{W}^{1,\infty}},\norm{\phi}_{\mathbf{W}^{1,\infty}}\right)$.
Then by comparing the defining identities \eqref{eq:ch3_volterra_reform} and \eqref{eq:ind_delta} at $t=t_n$, we have
\begin{align*}
    E^\delta_n \leq& \frac1{\phi_{\mathrm{min},0}} \left( \sigma + \norm{\bar{k}'}_{\mathbf{L}^\infty} E^\xi_{n,0} + \bar\delta \norm{\partial_x\phi}_{\mathbf{L}^\infty} E^\xi_{n,0} \right)\\
     &+ \frac{\norm{V}_{\mathbf{L}^\infty}\norm{\phi}_{\mathbf{L}^\infty}}{\phi_{\mathrm{min},0}} \sum_{m=0}^{n-1} \int_{t_m}^{t_{m+1}} (\sigma + K (\tau - t_m)) \d\tau \\
    &+ \frac{\norm{k(\cdot,0)}_{\mathbf{L}^\infty}}{\phi_{\mathrm{min},0}} \sum_{m=0}^{n-1} \int_{t_m}^{t_{m+1}}\Bigg[ \norm{\phi}_{\mathbf{L}^\infty}\frac{\norm{V'}_{\mathbf{L}^\infty}}{L} \left( K (\tau - t_m) + E^\delta_m \right)   \\
    &\quad +  \norm{V}_{\mathbf{L}^\infty}\big( \left( \norm{\partial_t\phi}_{\mathbf{L}^\infty}+ \norm{V}_{\mathbf{L}^\infty} \norm{\partial_x\phi}_{\mathbf{L}^\infty} \right) (\tau\!-\!t_m) + \norm{\partial_x\phi}_{\mathbf{L}^\infty} E^\xi_{n,m}\big)\Bigg] \d\tau \\
    & + \frac1{\phi_{\mathrm{min},0}} \sum_{m=0}^{n-1} \int_{t_m}^{t_{m+1}} \left( \norm{\partial_t\phi}_{\mathbf{L}^\infty}+\norm{V}_{\mathbf{L}^\infty} \norm{\partial_x\phi}_{\mathbf{L}^\infty} \right) \left( K (\tau - t_m) + E^\delta_m \right) \d\tau \\
    & + \frac{\norm{\delta}_{\mathbf{L}^\infty} + K\Delta t}{\phi_{\mathrm{min},0}} \sum_{m=0}^{n-1}  \Big( \left( \norm{\partial_{tx}^2\phi}_{\mathbf{L}^\infty} + \norm{V}_{\mathbf{L}^\infty} \norm{\partial_{xx}^2\phi}_{\mathbf{L}^\infty} \right)  E^\xi_{n,m} \\
   &\quad  + \frac{\norm{V'}_{\mathbf{L}^\infty}}{L} \norm{\partial_x\phi}_{\mathbf{L}^\infty} E^\delta_m + C\Delta t\Big) \Delta t,
\end{align*}
where in the last two lines we have used the identity
\begin{align*}
    \int_{t_m}^{t_{m+1}} \frac{\d\phi}{\d\tau} \left( \tau, \int_\tau^{t_n} v^{\delta}(\tilde{\tau})\d\tilde{\tau} \right) \d\tau = \phi(t_{m+1},\xi(t_n)-\xi(t_{m+1})) - \phi(t_m,\xi(t_n)-\xi(t_m)),
\end{align*}
and the constant $C=C\left( T,L,V,\norm{\phi}_{\mathbf{W}^{2,\infty}}\right)$. This inequality and \eqref{eq:error_analysis_xi} yield
\begin{align*}
    E^\delta_n \leq& C \left( \sigma + \Delta t + E^\xi_{n,0} + \sum_{m=0}^{n-1} E^\delta_m \Delta t + \sum_{m=0}^{n-1} E^\xi_{n,m} \Delta t\right) \\
    \leq& C \left( \sigma + \Delta t + \sum_{m=0}^{n-1} E^\delta_m \Delta t + \sum_{m=0}^{n-1} (m+1)E^\delta_m(\Delta t)^2 \right) ,
\end{align*}
which then gives 
\begin{equation}\label{eqn:trunc-err}
E^\delta_n\leq C(\sigma+\Delta t),\quad n=1,\cdots,N,
\end{equation} 
where the constant $C=C\left( T,L,V,\bar{\delta},\phi_{\mathrm{min},0}^{-1},\|f\|_{\mathbf{L}^\infty},\|{\bar{k}}\|_{\mathbf{W}^{1,\infty}},\|{\phi}\|_{\mathbf{W}^{2,\infty}}\right)$.
Consequently, we also have 
\begin{equation}\label{eqn:trunc-err-xi}
E^\xi_{n,m}\leq C(n-m)(\sigma+\Delta t)\Delta t,\quad 0\leq m \leq n \leq N.
\end{equation}

Now we give an error estimate on $|f(t_n)-f_n|$. We first note that \eqref{eq:ind_delta} yields
\begin{align*}
    \delta_{n+1} - &\delta_n = -\frac{\phi(t_n,0)}{\phi(t_{n+1},0)} V\left(\frac{\delta_n}{L}\right) k^\sigma(t_n,0) \Delta t \\
    &+ \frac1{\phi(t_{n+1},0)} \left( k^\sigma(t_{n+1},0) - k^\sigma(t_n,0) - \bar{k}(\xi_{n+1}) + \bar{k}(\xi_n) \right) \\
    &  - \frac1{\phi(t_{n+1},0)} \sum_{m=0}^{n} V\left(\frac{\delta_m}{L}\right) k^\sigma(t_m,0) (\phi(t_m,\xi_{n+1}-\xi_m) -\phi(t_m,\xi_n-\xi_m)) \Delta t \\
    & - \frac1{\phi(t_{n+1},0)} \sum_{m=0}^{n-1} (\delta_{m+1}-\delta_m) \left( \phi(t_{m+1},\xi_{n+1}-\xi_{m+1}) - \phi(t_{m+1},\xi_n-\xi_{m+1}) \right).
\end{align*}
By comparing \eqref{eq:ch3_delta_dot_rewrite} and the above identity, and using \eqref{eqn:trunc-err}-\eqref{eqn:trunc-err-xi}, we obtain
\begin{align*}
    &\frac1{\Delta t} \left| \delta_{n+1} - \delta_n - \int_{t_n}^{t_{n+1}} \frac{\d \delta(t)}{\d t}\d t \right| \\
    \leq&  C\left( \sigma+\Delta t+\frac{\sigma}{\Delta t} + \sum_{m=0}^{n-1} \frac1{\Delta t} \left| \delta_{m+1} - \delta_m - \int_{t_m}^{t_{m+1}} \frac{\d \delta(t)}{\d t}\d t \right| \right),
\end{align*}
and consequently
\begin{align*}
    \frac1{\Delta t} \left| \delta_{n+1} - \delta_n - \int_{t_n}^{t_{n+1}} \frac{\d \delta(t)}{\d t}\d t \right| \leq C\left( \sigma+\Delta t+\frac{\sigma}{\Delta t} \right),
\end{align*}
where the constant $C=C\left( T,L,V,\bar{\delta},\phi_{\mathrm{min},0}^{-1},\|f\|_{\mathbf{L}^\infty},\|{\bar{k}}\|_{\mathbf{W}^{2,\infty}},\|{\phi}\|_{\mathbf{W}^{2,\infty}}\right)$.
Then from the defining relation \eqref{eq:ind_f}, we deduce
\begin{align*} 
  \left| f_n - \frac1{\Delta t} \int_{t_n}^{t_{n+1}} f(t)\d t \right| \leq C\left( \sigma+\Delta t+\frac{\sigma}{\Delta t} \right).
\end{align*}
The right hand side takes its minimum when $\Delta t=\sigma^{1/2}$. 
Then with the choice $\Delta t=\sigma^{1/2}$, we obtain the final estimate:
$ \left| f_n - \frac1{\Delta t} \int_{t_n}^{t_{n+1}} f(t)\d t \right| \leq C \sigma^{1/2}$. 
This is summarized in the next result.

\begin{theorem}
Let \cref{assm:2} be fulfilled.
Let $\{f_n\}_{n=0}^{N-1}$ be the reconstruction obtained by the scheme \eqref{eq:ind_xi}-\eqref{eq:ind_delta}-\eqref{eq:ind_f}, and $f$ be the exact inflow rates. Then with the choice $\Delta t=\sigma^{1/2}$, the following error estimate holds:
$$ \sup\nolimits_{0\leq n\leq N-1} \left| f_n - \frac1{\Delta t} \int_{t_n}^{t_{n+1}} f(t)\d t \right| \leq C\sigma^{1/2}, $$
where the constant $C$ depends on $T,L,V,\bar{\delta},\phi_{\mathrm{min},0}^{-1},\|f\|_{\mathbf{L}^\infty},\|{\phi}\|_{\mathbf{W}^{2,\infty}},\|{\bar{k}}\|_{\mathbf{W}^{2,\infty}}$.
\end{theorem}

\begin{remark}\label{rmk:discret}
The proposed reconstruction method achieves the regularizing effect only via the time discretization, and hence the time step size $\Delta t$ has to be chosen properly. Either too large or too small $\Delta t$ can lead to large errors: a too small $\Delta t$ can significantly amplify the data noise, whereas a too large $\Delta t$ leads to large truncation errors. In practice, one may incorporate explicit regularization via suitable penalty \cite{ItoJin:2025} or suitably presmoothing the noisy data $k^\sigma(\cdot,0)$ via filtering \cite{LouisMaass:1990}. 
\end{remark}

\section{Numerical results and discussions}
\label{sec:exp}
In this section, we present several numerical experiments to illustrate the feasibility of recovering the inflow rates $f$ using the proposed explicit scheme \eqref{eq:ind_xi}-\eqref{eq:ind_delta}-\eqref{eq:ind_f}. 
We also briefly discuss one related inverse problem of recovering a time-independent inflow distribution $\phi$.

\subsection{Numerical results for inflow rate recovery}

In this subsection, we present a set of numerical experiments to illustrate the recovery of the inflow rates $f$. In these experiments, we assume that $k=1$ represents the maximum density (bump-to-bump traffic) and take $V(k)=1-k$. The direct problem is solved using an upwinding scheme with mesh size $\Delta x = \Delta t$. The exact observational data $k(t_n,0)$ is taken at $n=0,1,\ldots,N$, with $N=T/\Delta t$, and we generate noisy data $k^\sigma(t_n,0)$ by adding i.i.d. noise from the uniform distribution $\mathcal{U}[-\sigma,\sigma]$ to the exact data $k(t_n,0)$. We solve the inverse problem using the explicit scheme \eqref{eq:ind_xi}-\eqref{eq:ind_delta}-\eqref{eq:ind_f} applied on coarser temporal meshes than those for the direct problem, in order to avoid the uncontrolled amplification of the data noise.

\begin{example}\label{exp1}
We take $L=10$, $\bar k\equiv0$, and $\phi(t,x)= \frac1L\mathbf{1}_{[0,L]}(x)$, i.e., initially the road network is empty and the entering vehicles' remaining distances follow a uniform distribution on $[0,L]$.
We also set the inflow rates $f\equiv\bar f=0.15$, and consider two terminal times: (a) $T=8$ and (b) $T=16$.    
\end{example}

The difference of case (b) from (a) lies in the fact that in case (b) the discontinuity of $\phi$ at $x=L$ will propagate to the observation point $x=0$ within the time horizon $[0,T]$. The numerical results for cases (a) and (b) with exact data are shown in \cref{fig:exp_1}(i)--(ii), computed with $\Delta x = \Delta t= 10^{-4}$. One observes a fairly accurate reconstruction of the inflow rates $f$ in either case. Remarkably, the presence of the discontinuity does not influence much of the recovery accuracy. The results for noisy data (with $\sigma=10^{-4}$) are given in \cref{fig:exp_1}(iii). These results show that, with proper choice of $\Delta t$, the reconstructed $f$ is reasonably accurate. Although not presented, with a very small time step size, the reconstruction tends to be highly oscillatory, as predicted by \cref{rmk:discret}.

\begin{figure}[hbt!]
\centering\setlength{\tabcolsep}{0pt}
\begin{tabular}{ccc}
\includegraphics[width=.33\textwidth]{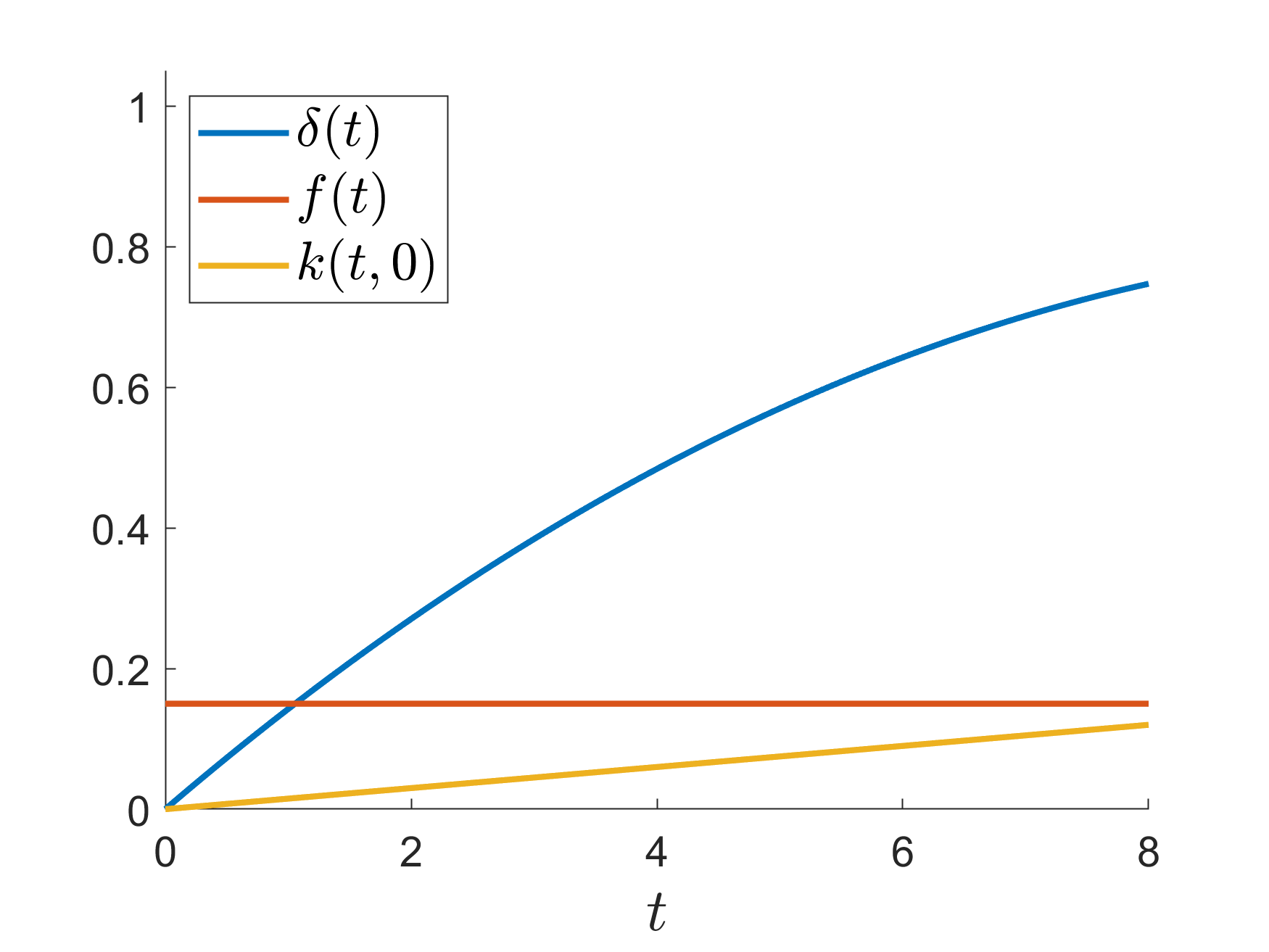}
&	\includegraphics[width=.33\textwidth]{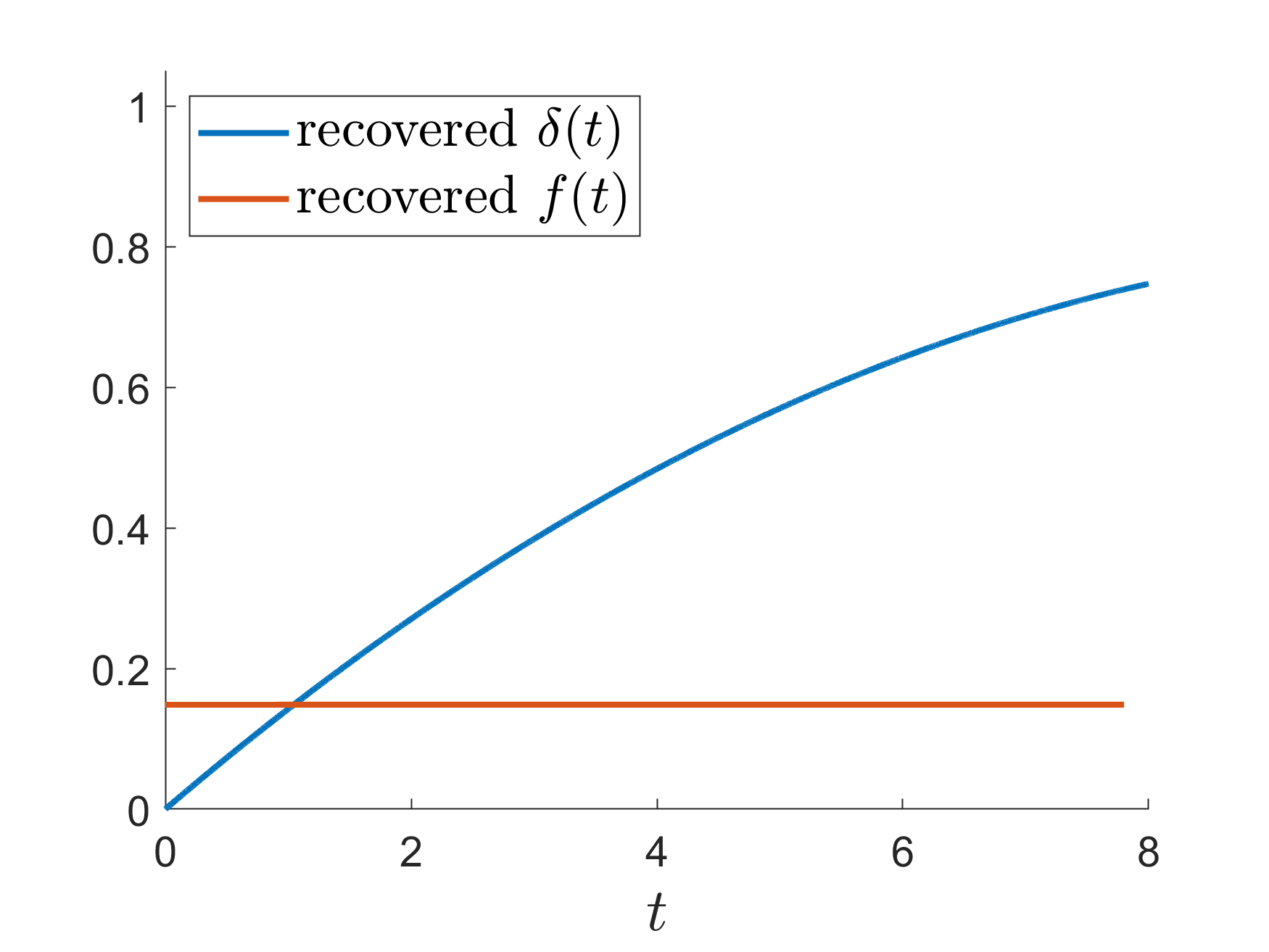} & \includegraphics[width=0.33\textwidth]{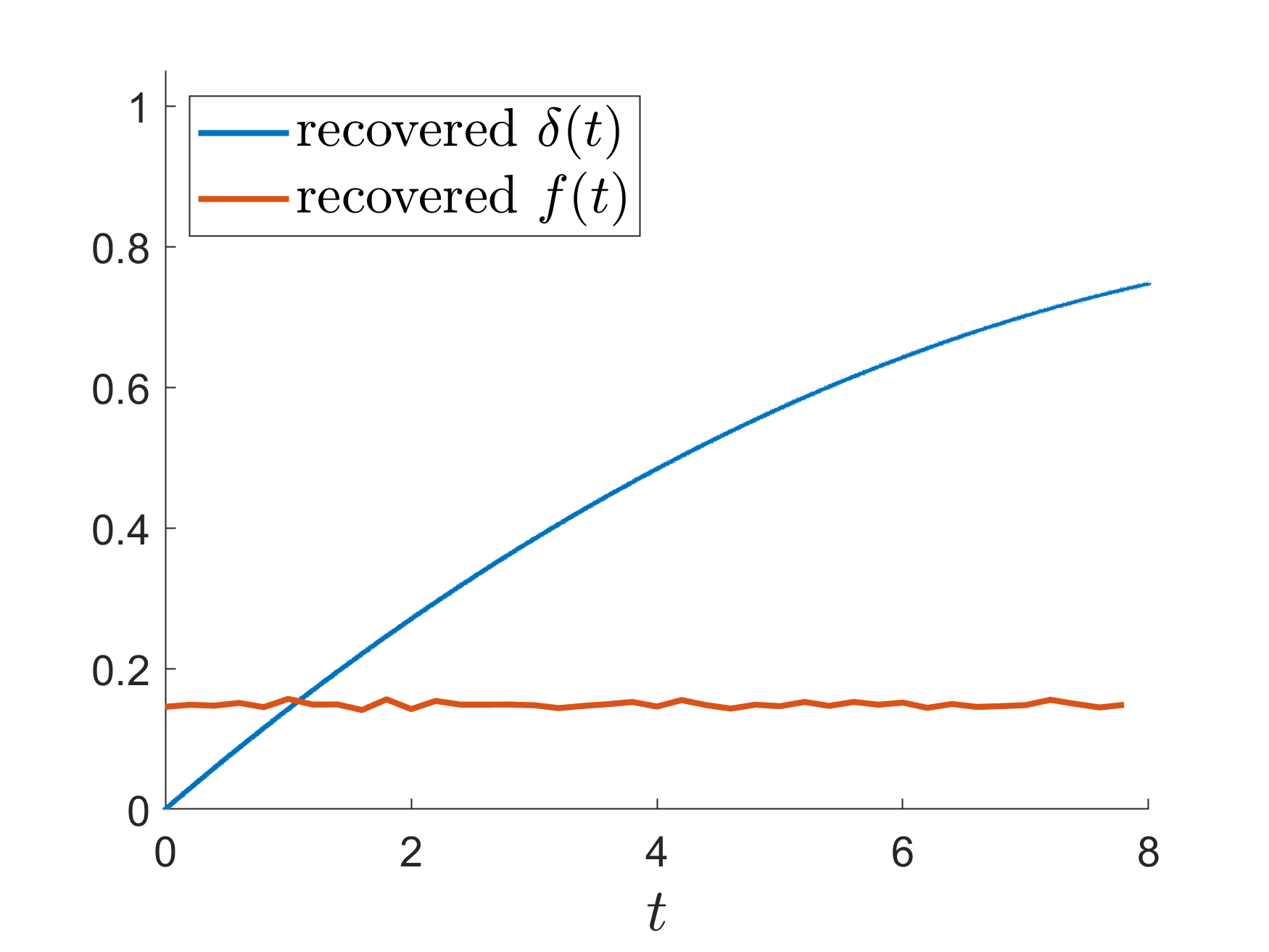}\\
\includegraphics[width=.33\textwidth]{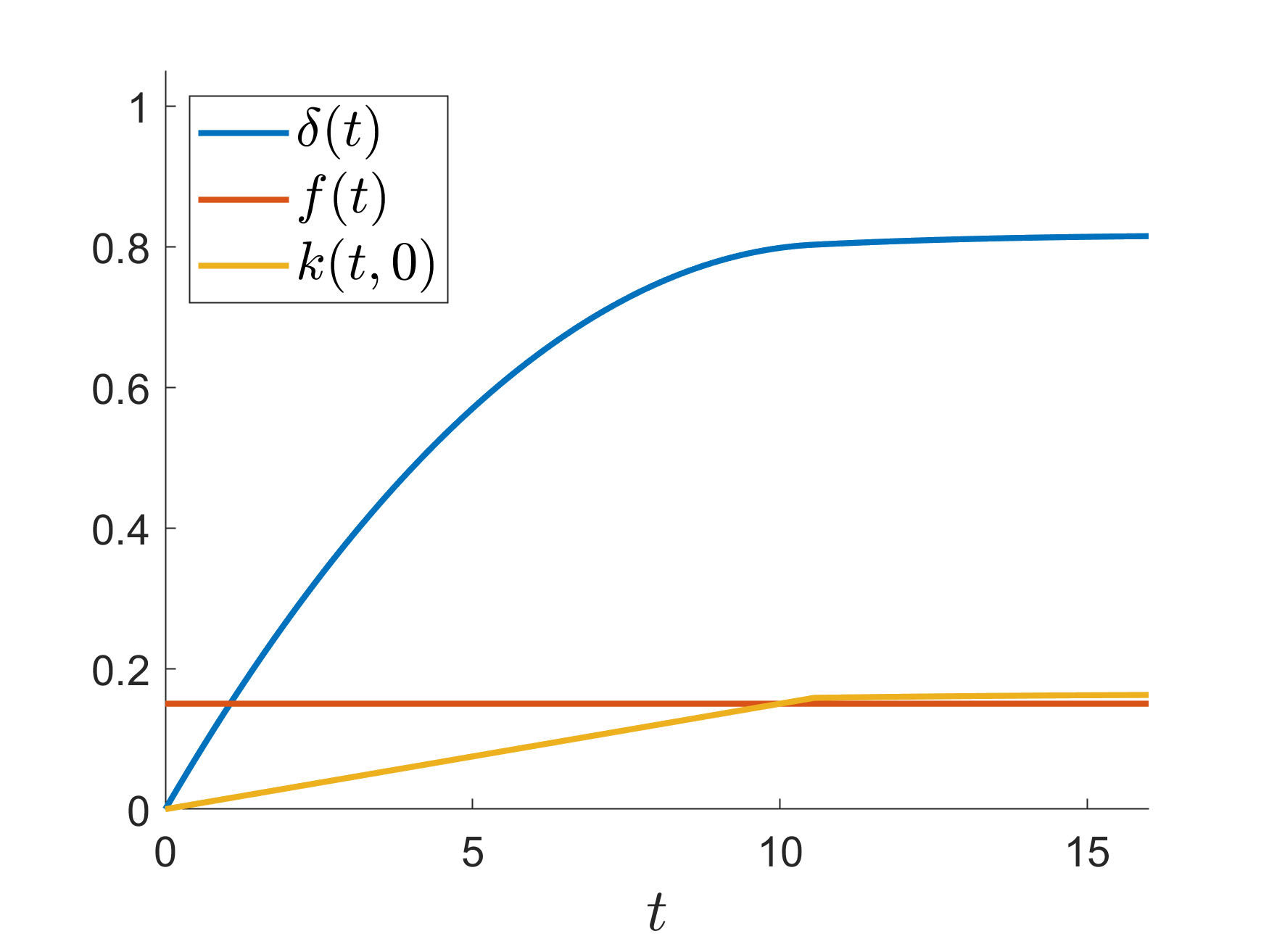}
&	\includegraphics[width=.33\textwidth]{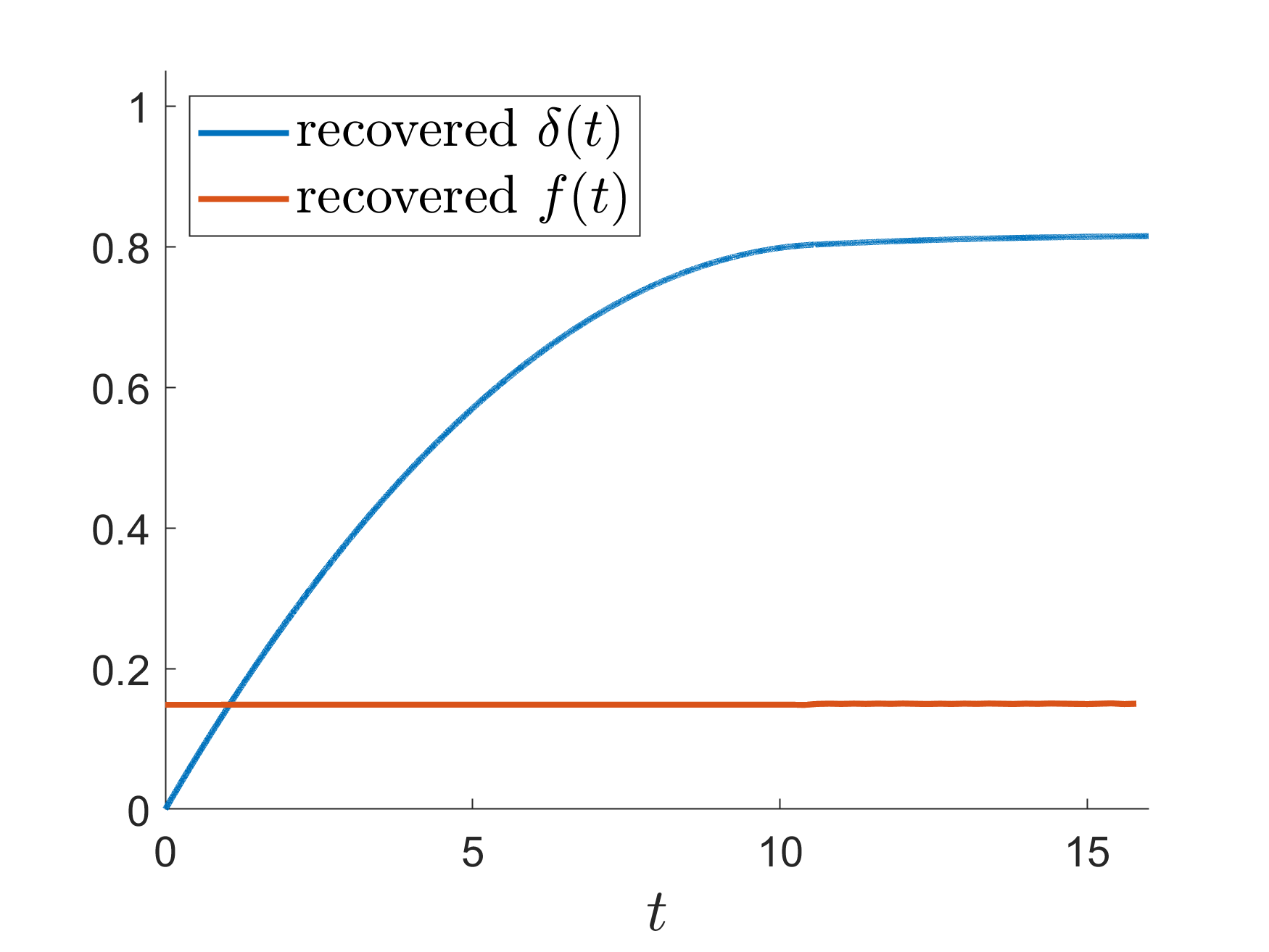} &\includegraphics[width=0.33\textwidth]{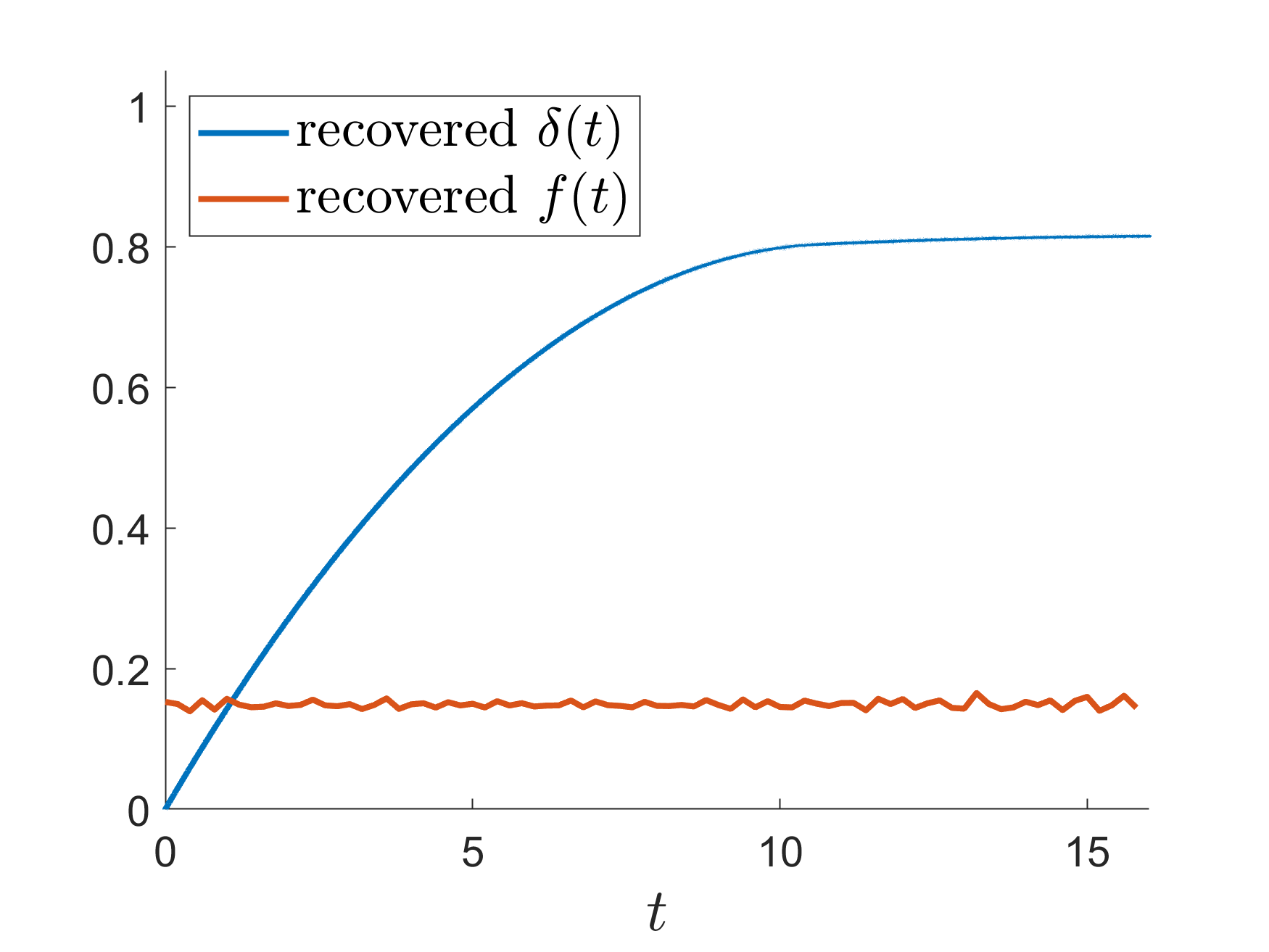}\\
(i) forward & (ii) inverse, exact & (iii) inverse, noisy
\end{tabular}
\caption{Numerical results for \cref{exp1} with exact data and noisy data. The left, mid and right panels are for the forward problem and inverse problem with exact and noisy data, respectively. The top and bottom rows are for cases (a) and (b), respectively.}
	\label{fig:exp_1}
\end{figure}

\begin{example}\label{exp2}
This example consists of three cases. In case (a), we take $L=10$, $T=8$, $\bar k\equiv0$, $\phi(x)=\frac1L\mathbf{1}_{[0,L]}(x)$, and $f(t)=0.2(1+\sin(2\pi t))$. In case (b), we use a nonzero initial data $\bar k(x)=0.1 e^{-10(x-L/2)^2}$ and keep other parameters in case (a) unchanged. In case (c), we use a time-dependent inflow distribution $\phi(t,x)=\frac{1}{\sqrt{2\pi} a}e^{\frac{-(x-b(t))^2}{2a^2}}$ with $a=0.4$ and $b(t)=\frac{L}{2}+\frac{2t-T}{4}$, and keep other parameters in case (a) unchanged.
\end{example}

The numerical results for the three cases with noisy data (with $\sigma=10^{-5}$) are shown in \cref{fig:exp_2}\,(i)--(iii).
Case (a) is to test the time-varying inflow rates $f$, and we observe that the method still gives fairly accurate reconstruction of $f$ with only small oscillations.
In case (b), initially the road network is nonempty and most vehicles have the remaining distance $L/2$.
In case (c), the inflow distribution $\phi$ is time-dependent and most entering vehicles at time $t$ have the remaining distance $b(t)=\frac{L}{2}+\frac{2t-T}{4}$, moving from $\frac{L}{2}-\frac{T}{4}$ to $\frac{L}{2}+\frac{T}{4}$ during the time horizon $[0,T]$. For cases (b) and (c), while the reconstructed $f$ is still as accurate as in case (a), the oscillation becomes larger in case (c). This is expected since the problem with time-dependent inflow distribution would be more challenging.

\begin{figure}[htbp!]
\centering\setlength{\tabcolsep}{0pt}
\begin{tabular}{ccc}
	\includegraphics[width=.33\textwidth]{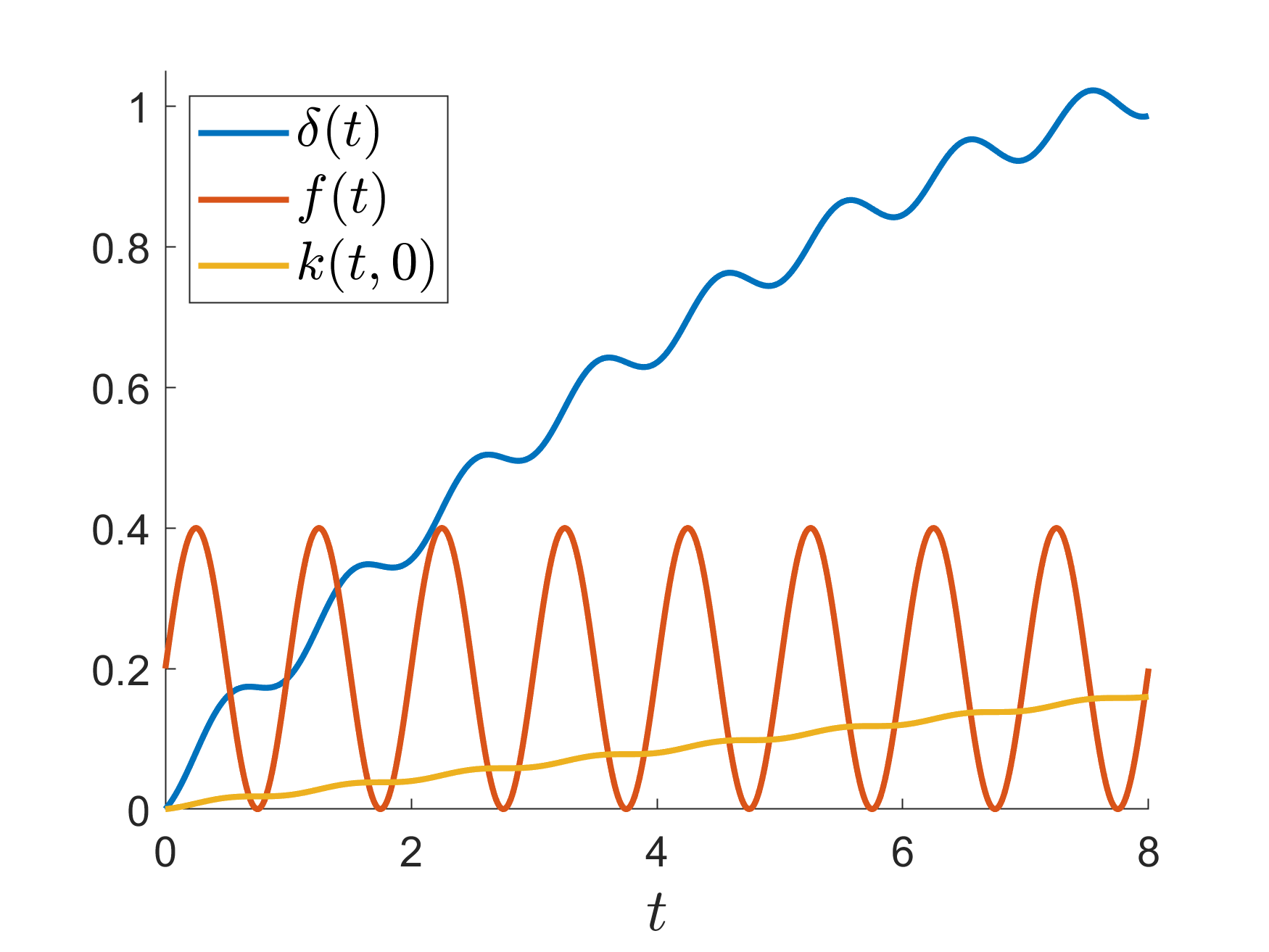} & \includegraphics[width=.33\textwidth]{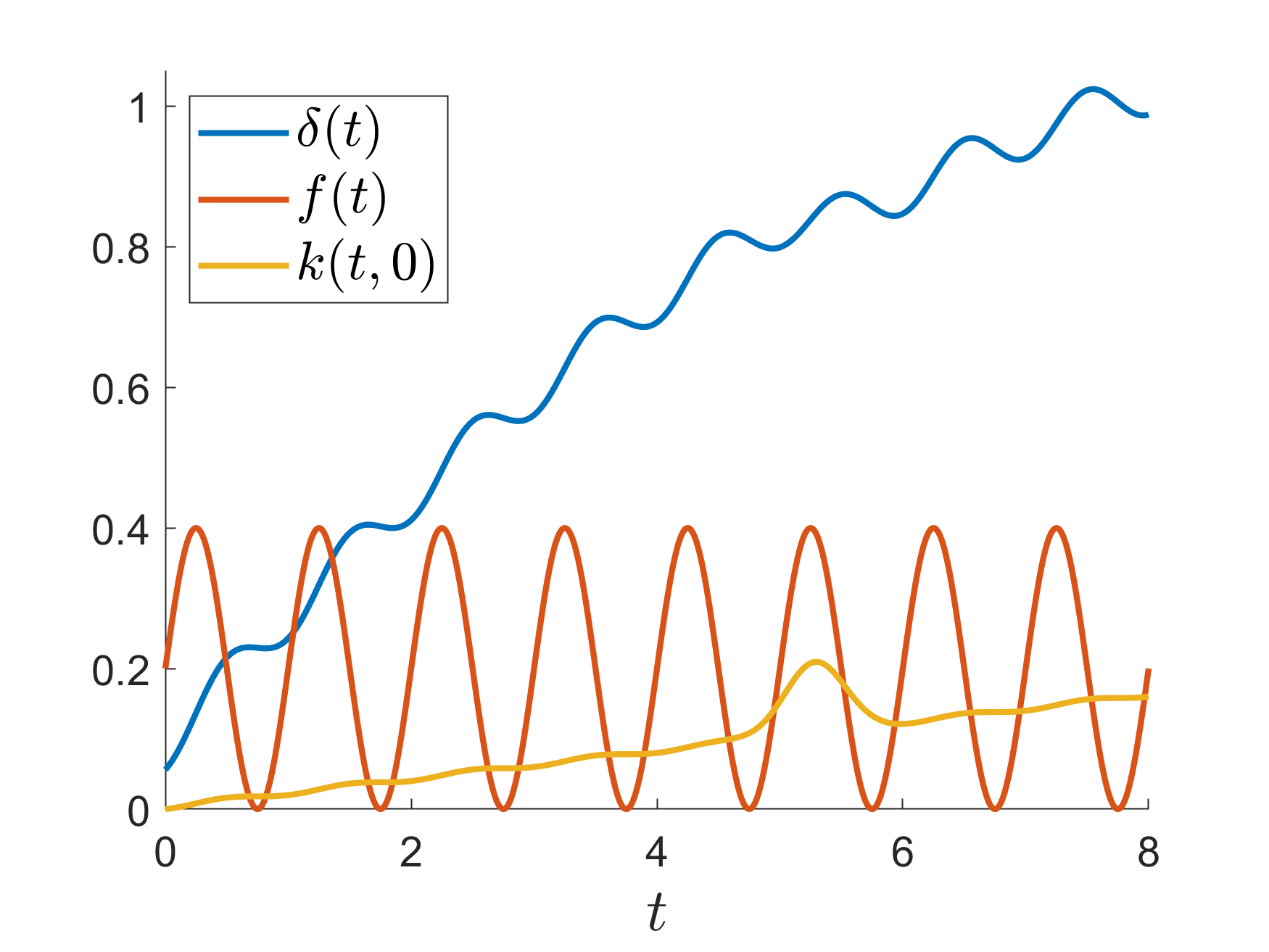} & \includegraphics[width=.33\textwidth]{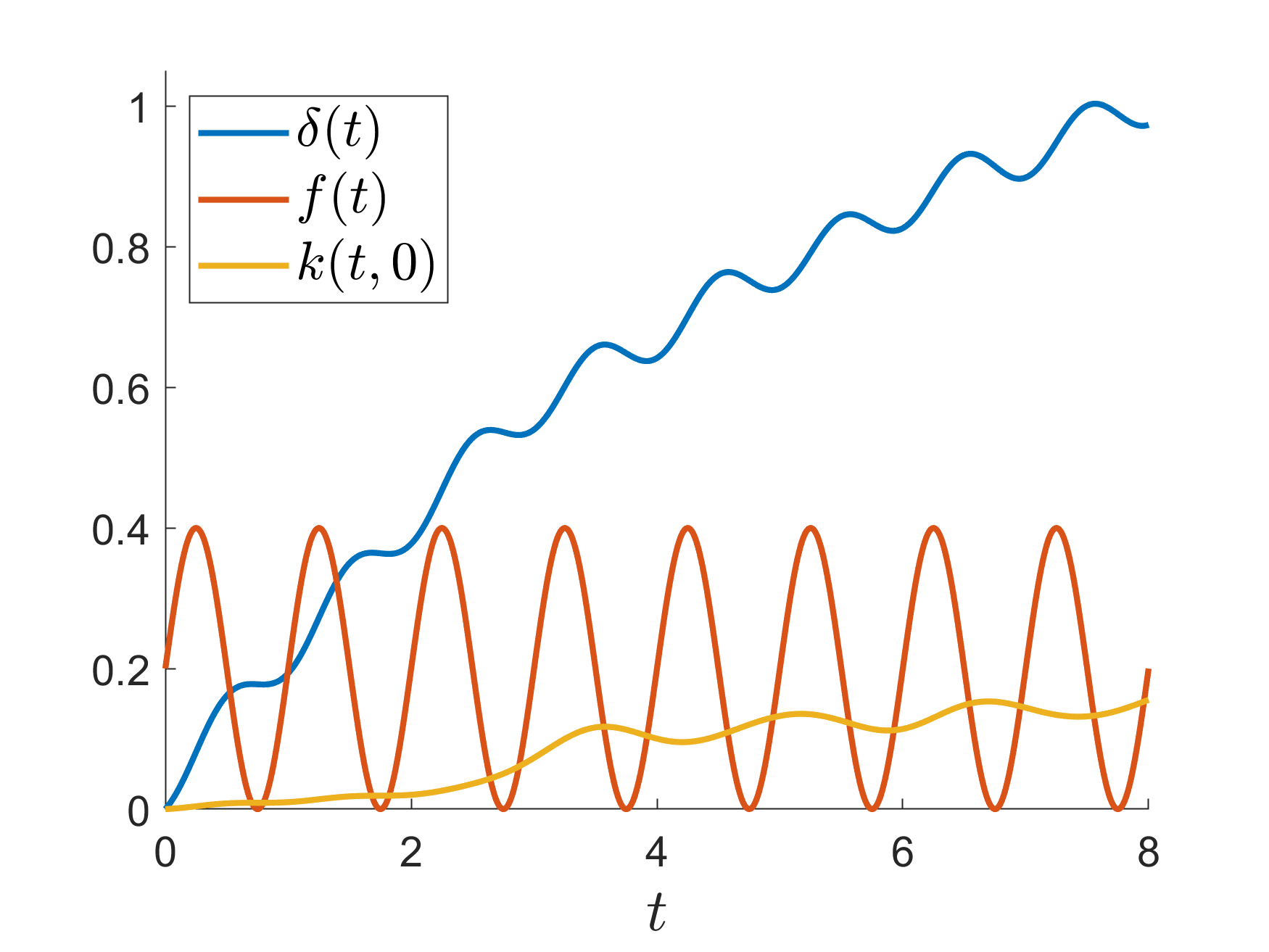} \\
	\includegraphics[width=.33\textwidth]{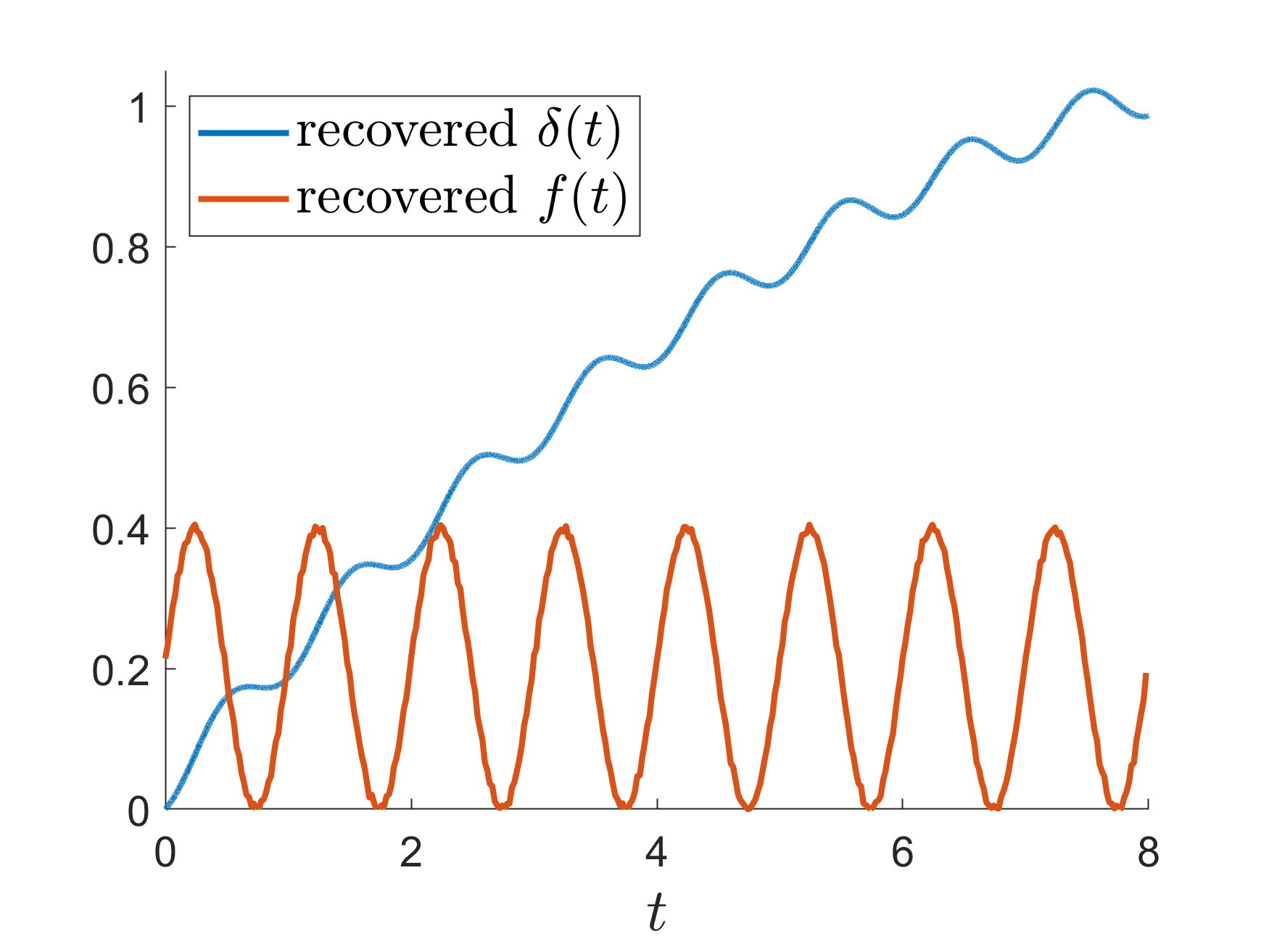} & \includegraphics[width=.33\textwidth]{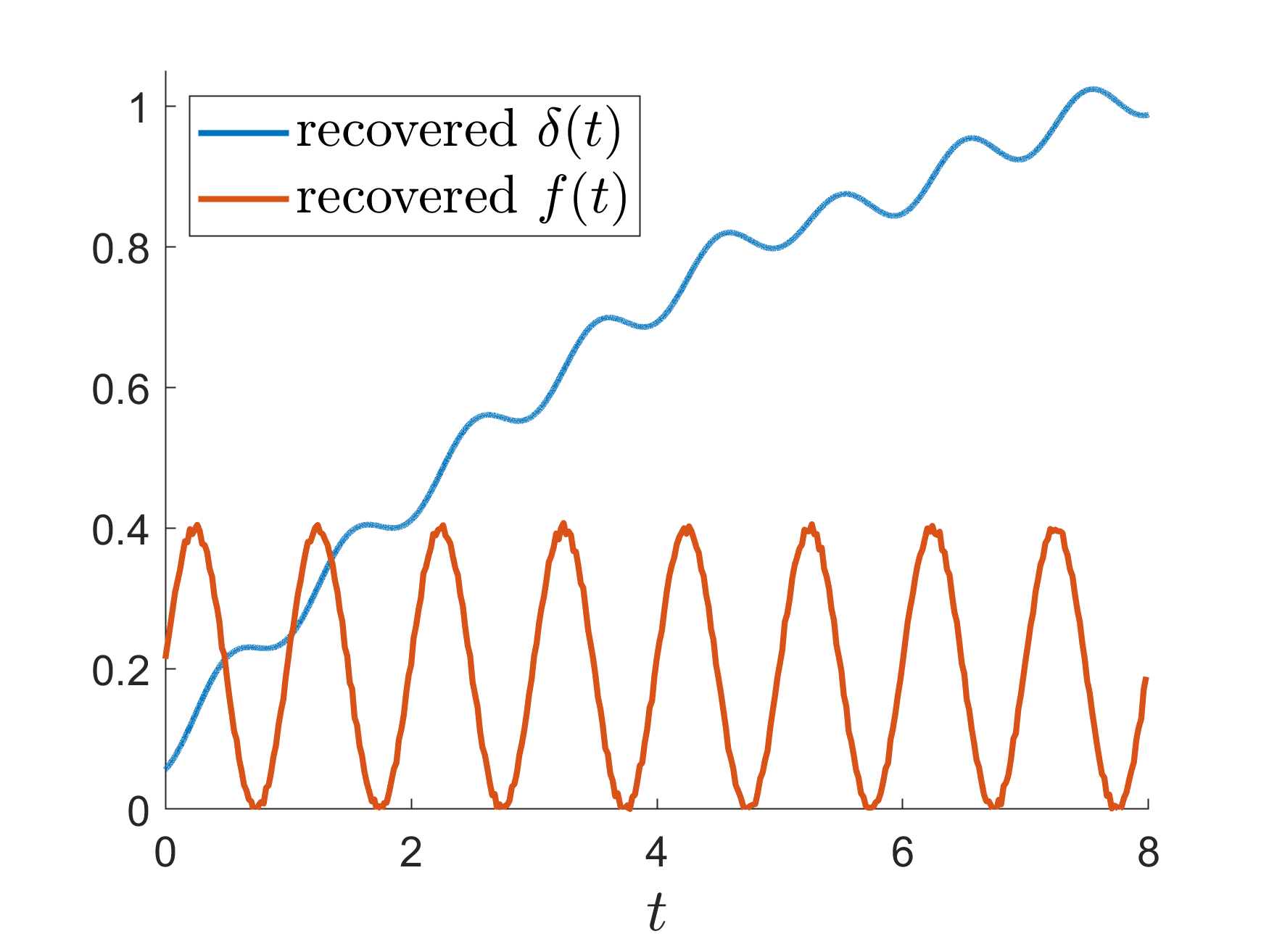} & \includegraphics[width=.33\textwidth]{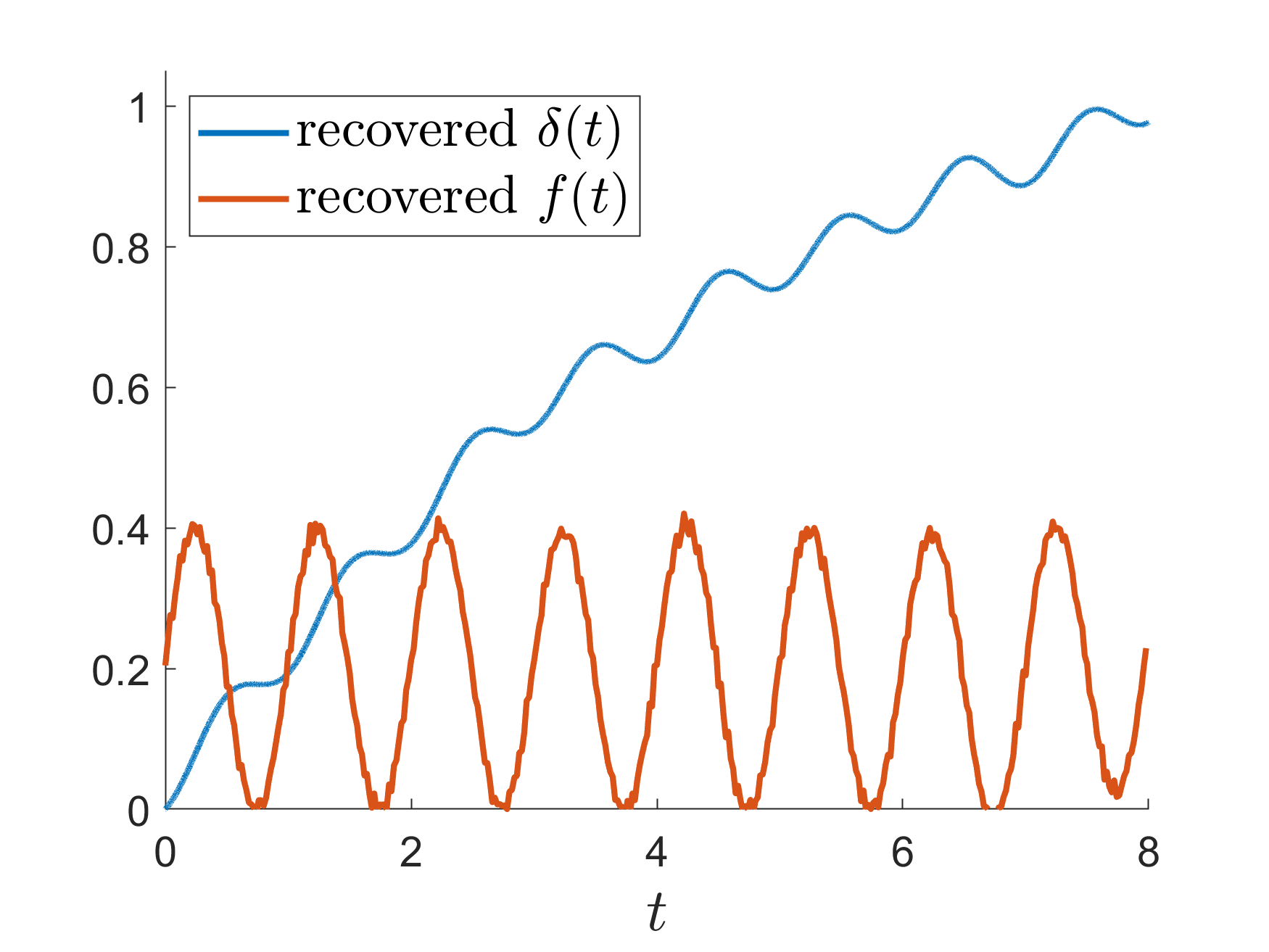} \\
 (i) case (a) & (ii) case (b)  & (iii) case (c)
\end{tabular}
	\caption{Numerical results for the three cases of \cref{exp2}. The top and bottom rows are for the forward problem and inverse problem with noisy data, respectively. }
	\label{fig:exp_2}
\end{figure}

\subsection{Inverse problem for inflow distribution}
\label{ssec:inverse2}
In this subsection, we discuss a related inverse source problem of calibrating the inflow distribution $\phi$ rather than inflow rates $f$ from the boundary data in a simple case where the inflow distribution $\phi$ is time independent.
More precisely, we make the following assumptions.
\begin{assumption}\label{assm:3}
{\rm(i)} The nonlocal kernel $w$ is a constant function, i.e., $ w(y)=\frac1L\mathbf{1}_{[0,L]}(y)$, for a given constant $L>0$; {\rm(ii)} the velocity function $V$ satisfies $V(k)\geq V_{\mathrm{min}}$ for all $k\geq0$ where $V_{\mathrm{min}}$ is a positive constant; {\rm(iii)} the initial data $\bar{k}\in\mathbf{L}^\infty([0,\infty);[0,\infty))$ is $\mathbf{C}^2$ smooth and supported on $[0,L]$; {\rm(iv)} the inflow distribution $\phi(t,x)\equiv\phi(x)$ where $\phi\in\mathbf{L}^\infty([0,\infty);[0,\infty))$ is $\mathbf{C}^2$ smooth and supported on $[0,L]$; and {\rm(v)} the inflow rates $f$ is Lipschitz continuous and $f(0)>0$. 
\end{assumption}
% we assume that the inflow distribution $\phi(t,x)\equiv\phi(x)$ and that the inflow rates $f(t)$ for $t\in[0,T]$ are given. 
We aim to solve the following inverse problem of the model \eqref{eq:bathtub_model}: given the observation of the boundary data $k(t,0)$ for $t\in[0,T]$, 
find the inflow distribution $\phi(x)$ for $x\in[0,L^*]$ for some $L^*>0$.
The starting point of the analysis is the identities \eqref{eq:inv_source_1}-\eqref{eq:inv_source_3}. Since the inflow rates $f(t)$ and the boundary data $k(t,0)$ are known for $t\in[0,T]$, we can integrate the ODE \eqref{eq:inv_source_1} to obtain $\delta(t)$ for $t\in[0,T]$ and obtain $\xi(t)$ for $t\in[0,T]$ from \eqref{eq:inv_source_3}.
Then, by changing variables $z=\xi(t)$ and $y=\xi(t)-\xi(\tau)$, we can reformulate \eqref{eq:inv_source_2} as
\begin{align}\label{eqn:inv-phi}
    k\left(\xi^{-1}(z),0\right) = \bar{k}(z) + \int_0^z \frac{f\left(\xi^{-1}(z-y)\right)}{V\left(\delta\left(\xi^{-1}(z-y)\right)/L\right)} \phi(y) \d y,
\end{align}
for $z\in[0,L^*]$ where $L^*=\xi(T)$.
This is a Volterra integral equation of the first kind for the function $\phi$.
Following the known theory \cite{Lamm:2000}, Volterra integral equations of the first kind are ill-posed, and the degree of ill-posedness depends on the properties of the kernel. Indeed, using the fact $\xi(0)=0$ and the Lipschitz continuity of $f$ and $V$, we obtain
\begin{equation*}
    \lim_{y\to 0^+} \frac{f\left(\xi^{-1}(y)\right)}{V\left(\delta\left(\xi^{-1}(y)\right)/L\right)} = \frac{f(0)}{V(\bar{\delta}/L)}>0.
\end{equation*} 
Thus, assuming good differentiability of $k$, $\bar k$ and $\delta$, and then differentiating both sides of \eqref{eqn:inv-phi} yield
\begin{equation*}
    \frac{\d}{\d z}\left(k\left(\xi^{-1}(z),0\right)-\bar k(z)\right) = \frac{f(0)}{V(\bar\delta/L)} \phi(z) + \int_0^z \frac{\d}{\d z}\left( \frac{f\left(\xi^{-1}(z-y)\right)}{V\left(\delta\left(\xi^{-1}(z-y)\right)/L\right)} \right)\phi(y) \d y.
\end{equation*} 
This is a linear Volterra integral equation of the second kind for $\phi$, which is well-posed within $\mathbf{L}^p$ spaces given appropriate regularity conditions \cite{Lamm:2000}. Therefore, under \cref{assm:3}, the inverse problem roughly amounts to differentiating the problem data $k(t,0)$ once, resulting in a problem that is only mildly ill-posed. A thorough analysis of this is beyond the scope of the current work and is reserved for future investigation. Instead, we proceed to conduct a numerical exploration of its feasibility.

For the numerical reconstruction, we decompose it into two steps: (i) solve for $\xi^{-1}(z)$ and (ii) solve for $\phi(z)$.
In the first step, we take a uniform temporal mesh $0 = t_0 < t_1 < \cdots < t_{N_t-1} < t_{N_t} = T$, $ t_n = n\Delta t$,
with the mesh size $\Delta t = T/N_t$.
Then we discretize \eqref{eq:inv_source_1} and \eqref{eq:inv_source_3} as:
\begin{align*}
    \delta_{n+1} &= \delta_n + \Delta t\left( f(t_n) - V\left(\frac{\delta_n}L\right)k(t_n,0) \right), \quad n\geq0; \quad \delta_0 = \bar{\delta}. \\
    \xi_n &= \sum_{m=0}^{n-1} V\left(\frac{\delta_m}{L}\right) \Delta t, \quad n\geq1; \quad \xi_0 = 0.
\end{align*}
We solve for $\{\xi_n\}_{n=0}^{N_t}$ from the above equations.
Then we take a uniform spatial mesh $0 = x_0 < x_1 < \cdots < x_{N_x-1} < x_{N_x} = \norm{V}_{\mathbf{L}^\infty}T$, $ x_j = j\Delta x$,
with the mesh size $\Delta x = \norm{V}_{\mathbf{L}^\infty}T/N_x$.
For any $x_j\in(0, \xi_{N_t}]$, there is a unique $(\xi_n,\xi_{n+1}]$ such that $x_j\in (\xi_n,\xi_{n+1}]$. Then for all $j$, we define $\tau_j$, $f_j$ and $v_j$ respectively by
\begin{align*}
    &\tau_j = \frac{\xi_{n+1} - x_j}{\xi_{n+1} - \xi_n} t_n + \frac{x_j - \xi_n}{\xi_{n+1} - \xi_n} t_{n+1}, \qquad
    f_j = \frac{\xi_{n+1} - x_j}{\xi_{n+1} - \xi_n} f(t_n) + \frac{x_j - \xi_n}{\xi_{n+1} - \xi_n} f(t_{n+1}) \\
    &\mbox{and} \qquad
    v_j = \frac{\xi_{n+1} - x_j}{\xi_{n+1} - \xi_n} V(\delta_n/L) + \frac{x_j - \xi_n}{\xi_{n+1} - \xi_n} V(\delta_{n+1}/L).
\end{align*}
In the second step, we solve for $\{\phi_l\}_{l=1}^{N_x}$, which collectively reconstructs $\phi(x)$ for $x\in[0,\xi(T)]$, from the following equation:
\begin{align*}
    k(\tau_j, 0) = \bar{k}(x_j) + \sum_{l=1}^j \frac{f_{j-l}}{v_{j-l}}\phi_l, \quad j=1,2,\cdots,N_x.
\end{align*}

Last we illustrate the feasibility of the approach with one example.
\begin{example}\label{exp4}
We take $L=10$, $\bar{k}\equiv0$, and $\phi(x)=\frac1L\mathbf{1}_{[0,L]}(x)$. We set the inflow rates to be $f\equiv\bar{f}=0.15$, and consider two cases: (a) $T=8$; (b) $T=16$.
\end{example}

The settings of the example are identical to those in \cref{exp1}. The results for \cref{exp4} with noisy data (with $\sigma=10^{-4}$) are shown in \cref{fig:exp_4}. In both cases, we observe that the reconstructions are fairly accurate, with only very mild oscillations. These results clearly show the feasibility of the approach for recovering the inflow distribution $\phi$ from the lateral boundary data $k(\cdot,0)$.

\begin{figure}[htbp]
\centering\setlength{\tabcolsep}{0pt}
\begin{tabular}{cc}
    \includegraphics[width=.48\textwidth]{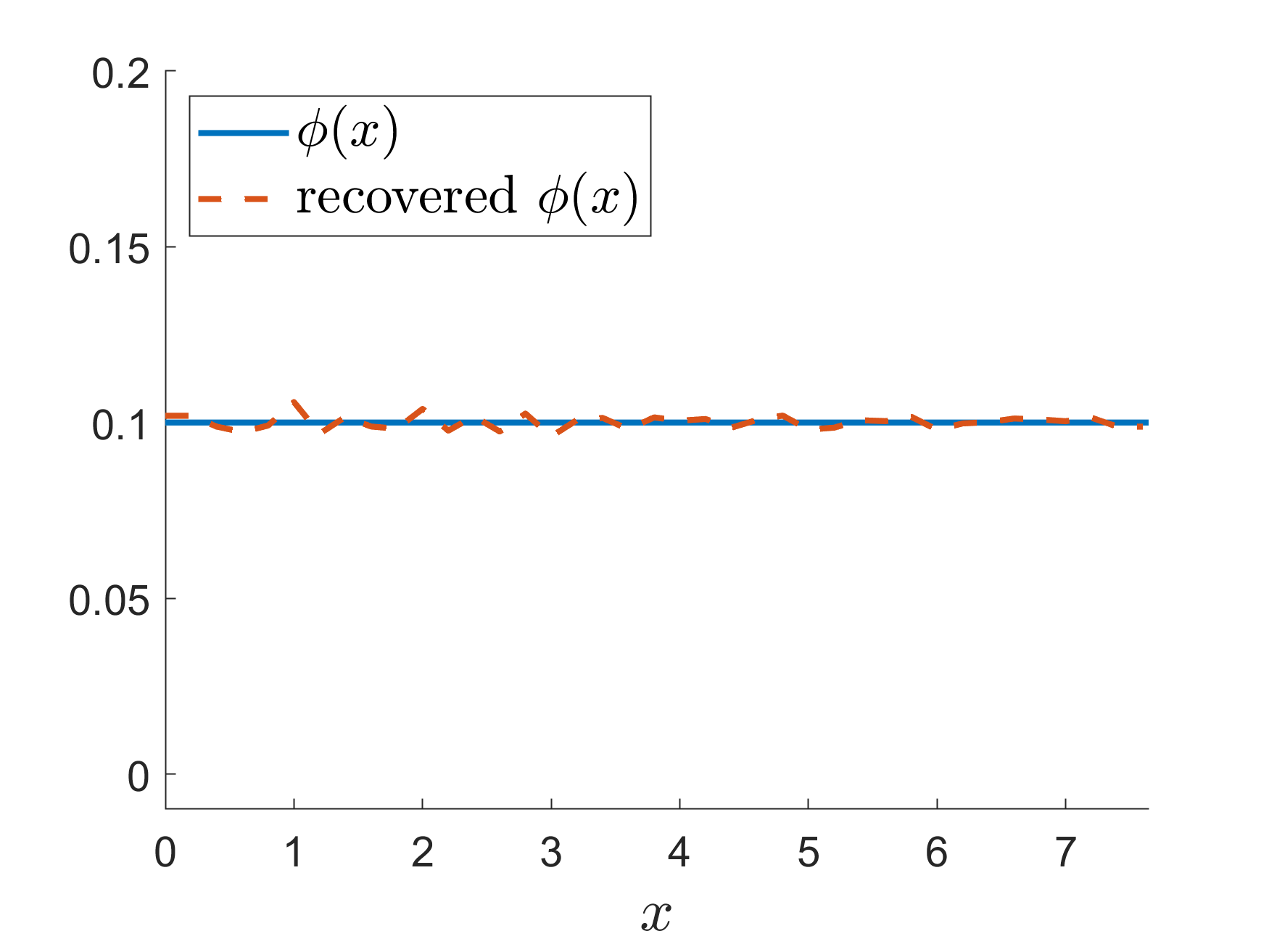} & \includegraphics[width=.48\textwidth]{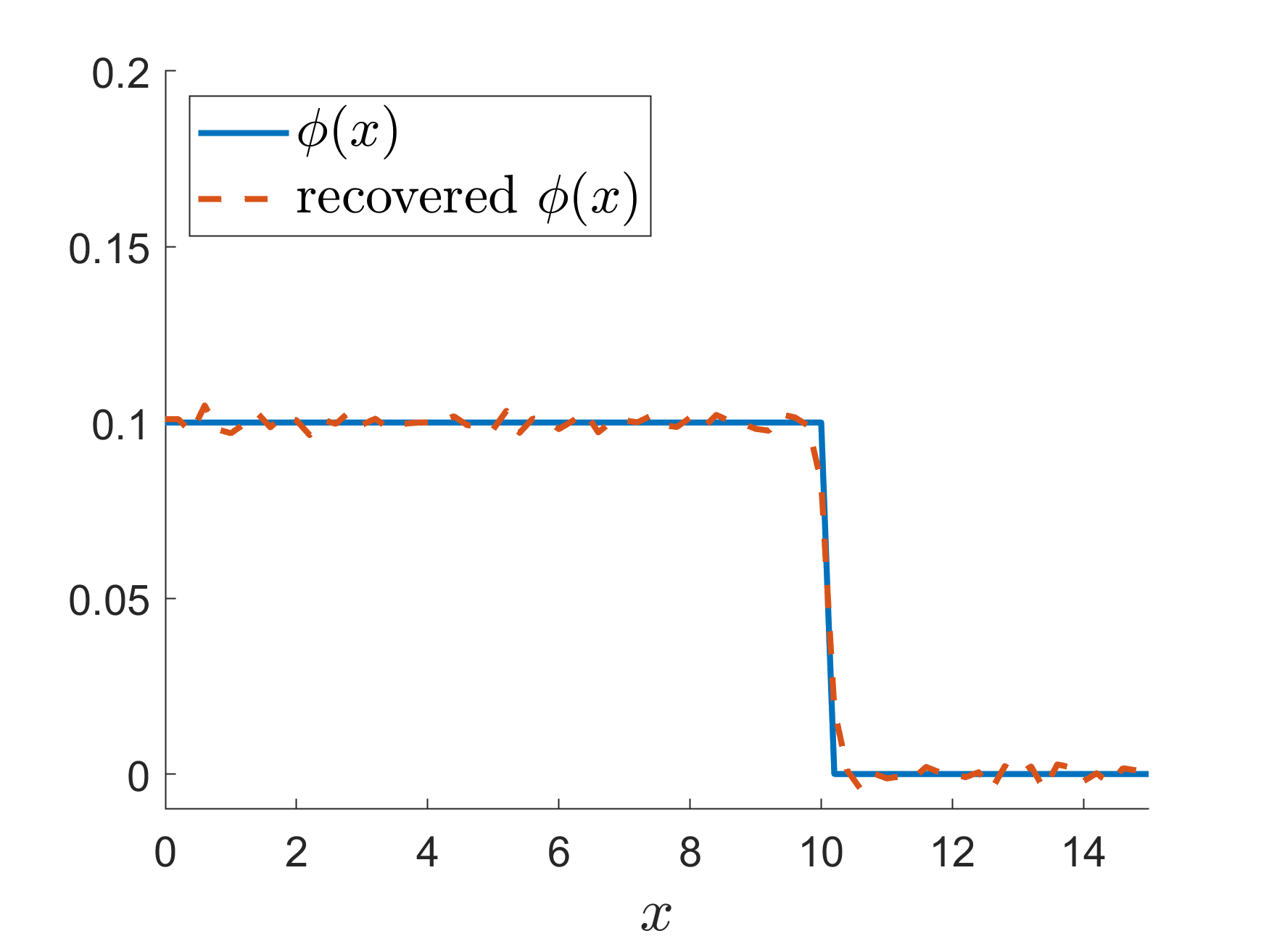} \\
    (i) case (a) & (ii) case (b)
\end{tabular}
\caption{Numerical results for the two cases of \cref{exp4} with noisy data. }
\label{fig:exp_4}
\end{figure}

\section{Concluding remarks}
In this work, we have investigated the generalized bathtub model arising in network traffic flows. We have established the well-posedness of the direct problem and a conditional Lipschitz stability for an inverse source problem of recovering the inflow rates from lateral boundary measurements. This stability result is due to the transport nature of the model with suitable smoothness conditions on the inflow distribution. We have also discussed the existence issue when the inflow distribution is nonsmooth. Furthermore, we presented an explicit numerical method for recovering the inflow rates, provided the error analysis, and conducted multiple numerical experiments showcasing the feasibility of the method. To the best of our knowledge, this represents one first mathematical study of an inverse source problem for nonlocal traffic flow models.
The theoretical and numerical findings in this work may offer useful insights for studies on inverse problems in other nonlocal traffic flow models, or more broadly, in nonlocal transport equations and balance laws.

In future work, we would like to explore more on the case of nonsmooth inflow distribution under some realistic traffic scenarios. In addition, it is of much interest to implement the reconstruction algorithm with field data and validate the applicability of the generalized bathtub model in real traffics. We are also interested in other inverse problems associated with the generalized bathtub model, e.g., calibrating both inflow rates and inflow distribution from more comprehensive measurement data.

% \section*{Acknowledgments}

\bibliographystyle{siam}
\bibliography{references}
\end{document}